\documentclass[12pt]{amsart}
\usepackage{amsthm}
\usepackage{amsmath}
\usepackage[margin=1.25in]{geometry}
\usepackage{amssymb}
\usepackage{verbatim}
\usepackage{longtable,enumitem,tikz-cd,color}
\usepackage{stmaryrd}
\usepackage{graphicx}
\usepackage{latexsym}   
\usepackage{amsmath}    
\usepackage{amsbsy}
\usepackage{array}
\usepackage[all]{xy}
\usepackage{hyperref}
\usepackage{caption}

\theoremstyle{plain}
\newtheorem{algorithm}{Algorithm}[section]

\newtheorem{definition}[algorithm]{Definition}
\newtheorem{observe}[algorithm]{Observation}
\newtheorem{example}[algorithm]{Example}

\newtheorem{lemma}[algorithm]{Lemma}

\newtheorem{main}{Theorem}

\newtheorem*{theorem*}{Theorem}
\newtheorem{theorem} [algorithm] {Theorem}

\newtheorem{proposition}[algorithm]{Proposition}
\newtheorem{remark}[algorithm]{Remark}

\numberwithin{equation}{algorithm}

\newtheorem*{MSR}{Maximal Symmetry Rank Theorem}
\newtheorem*{1/2MSR}{Half-Maximal Symmetry Rank Theorem}
\newtheorem*{Codim1}{Codimension One Lemma}
\newtheorem*{Codim2}{Codimension Two Lemma}
\newtheorem*{Codim3}{Codimension Three Lemma}
\newtheorem*{Codim4}{Codimension Four Lemma}
\newtheorem*{Borel}{Borel Formula}

\newtheorem*{CL}{Connectedness Lemma}
\newtheorem*{PL}{Periodicity Lemma}

\newcommand{\ep}{\epsilon}

\newcommand{\Z}{\mathbb{Z}}
\newcommand{\R}{\mathbb{R}}\newcommand{\C}{\mathbb{C}}\newcommand{\HH}{\mathbb{H}}

\newcommand{\RP}{\mathbb{R}\mathrm{P}}
\newcommand{\CP}{\mathbb{C}\mathrm{P}}
\newcommand{\HP}{\mathbb{H}\mathrm{P}}
\newcommand{\FP}{\mathbb{F}\mathrm{P}}

\newcommand{\gO}{\ensuremath{\operatorname{\mathsf{O}}}}
\newcommand{\Sp}{\ensuremath{\operatorname{\mathsf{Sp}}}}
\newcommand{\gU}{\ensuremath{\operatorname{\mathsf{U}}}}

\DeclareMathOperator{\diag}{diag}

\newcommand{\of}[1]{\left(#1\right)}
\newcommand{\ceil}[1]{\left\lceil #1 \right\rceil}
\newcommand{\floor}[1]{\left\lfloor #1 \right\rfloor}

\newcommand{\st}{~|~}

\newcommand{\embedded}{\hookrightarrow}

\newcommand{\codim}{{\operatorname{codim\,}}}                      
\usepackage[T1]{fontenc}
\usepackage{makecell}

\title{Positive curvature and discrete abelian symmetry} 
\author{Lee Kennard}\address{\hspace{-.1in}Department of Mathematics, Syracuse University, Syracuse, NY 13244 USA}\email{ltkennar@syr.edu}
\author{Elahe Khalili Samani}\address{\hspace{-.1in}Department of Mathematics, Clark University, Worcester, MA 01610 USA}
\email{ekhalilisamani@clarku.edu}
\author{Catherine Searle}\address{\hspace{-.1in}Department of Mathematics, Wichita State University, Wichita, KS 67260 USA}\email{searle@math.wichita.edu}
\date{\today}

\begin{document}
\begin{abstract}  

By replacing the torus with an elementary abelian two-group, 
we generalize Grove and Searle's maximal symmetry result  
 and Wilking's half-maximal symmetry result 
 for positively curved manifolds with an isometric torus action.
\end{abstract}

\maketitle\smallskip

The classification of closed manifolds of positive curvature is a long-standing problem in Riemannian geometry. To date, other than some special examples in  dimensions less than or equal to $24$, all known examples are spherical in nature and are highly symmetric. Over the last 30 years, the additional assumption of  symmetries has been a
fruitful approach to such a classification.  In particular, since every continuous group contains a maximal torus, the case of torus actions has naturally attracted a great deal of attention. In turn, each torus is a product of circles and therefore contains a product of cyclic groups. Previous work on positively and non-negatively curved manifolds  with discrete symmetries of dimension $4$ can be found in Yang \cite{Yan94}, Hicks \cite{Hic97}, Fang \cite{Fan08}, and Kim and Lee \cite{KL09}, and of higher dimensions in Fang and Rong \cite{FR04},  Su and Wang \cite{SW08},  and Wang \cite{W10}. For all of these results, the focus has  been on choosing a cyclic group  or a product of cyclic groups, where the order of each cyclic group is  large enough to produce a fixed point.

In this paper, we focus on products of cyclic groups
 of order two. If the torus under consideration has rank $r$, then this subgroup is an elementary abelian two-group of $\Z_2$-rank $r$. We will denote it by $\Z_2^r$ and call it a $\Z_2$-torus. 
For reasons we discuss below, and unlike the hypotheses of the above-mentioned work on discrete symmetries, all of our results will require the assumption that the $\Z_2$-torus has a fixed point. Our first result is an easy consequence of the work of Fang and Grove in \cite{FG} on reflections in positive and non-negative curvature. 

\begin{main}\label{thm:n}
Let $M^n$ be a closed, positively curved manifold such that $\Z_2^r$ acts isometrically on $M$ with a non-empty fixed-point set. If $r \geq n$, then $r=n$ and  $M$ is equivariantly diffeomorphic to $S^r$ or $\RP^r$  with a linear $\Z_2^r$-action.
\end{main}

We remark that Theorem \ref{thm:n} is analogous to the following maximal symmetry rank result of Grove and Searle \cite{GS}, where the {\em symmetry rank} of a Riemannian manifold  equals the rank of its isometry group.

\begin{MSR}\label{MSR}\cite{GS} 
Let $M^n$ be a closed, positively curved manifold admitting an isometric and effective $T^r$-action. If $r \geq \tfrac{n}{2}$, then $r=\lceil \tfrac{n}{2}\rceil$ and $M^n$ is equivariantly diffeomorphic to $S^{2r}$, $\RP^{2r}$, $\CP^{r}$, or $S^{2r-1}/\Z_k$ for some $k \geq 3$ with a linear $T^r$-action.
\end{MSR}

We define the $\Z_2$-{\em fixed point symmetry rank} of a Riemannian manifold to be the rank of the largest $\Z_2$-torus contained in its isometry group whose action has a fixed point. Both Theorem \ref{thm:n} and the \hyperref[MSR]{Maximal Symmetry Rank Theorem} determine the corresponding maximal $\Z_2$-fixed point symmetry rank and maximal symmetry rank, respectively, of a positively curved manifold and prove rigidity up to equivariant diffeomorphism. The next result can also be viewed as an analog of the \hyperref[MSR]{Maximal Symmetry Rank Theorem}, as it includes the remaining manifolds having maximal torus symmetry, although only up to homotopy equivalence.

\begin{main}\label{thm:nover2}
Assume $\Z_2^r$ acts effectively and isometrically on a closed, positively curved manifold $M^n$ with a non-empty fixed-point set. If $n\not \in\{3, 4, 7, 11, 12, 23\}$ and
	\[r \geq \frac{n+1}{2},\] 
then one of the following holds:
	\begin{enumerate}
	\item $M$ is homotopy equivalent to $S^n$ or $\RP^n$; or
	\item 	$M$ is homotopy equivalent to $\CP^{r-1}$ or $S^{2r-1}/\Z_k$, with $k\geq 3$, and $r=\lceil \tfrac{n+1}{2}\rceil$.
		\end{enumerate}
\end{main}
 
We note that both spaces in Conclusion (2) can be realized, and we show this in Example \ref{ex1}. For dimension $n = 3$, the result is not true since there are spherical space forms $S^3/\Gamma$ with non-cyclic fundamental group $\Gamma$ that admit $\Z_2^2$-symmetry with fixed point. Similarly, for dimensions $n = 6$ and $n = 8$, the symmetry assumption cannot be relaxed, as we see in Example \ref{1/4}. Finally, we note that for the excluded values of $n$ in Theorem \ref{thm:nover2}, we prove a slightly weaker result in Theorem \ref{thm:nover2PLUS}.

The conclusion of Theorem \ref{thm:nover2} is similar to that of the \hyperref[MSR]{Maximal Symmetry Rank Theorem}. However, the proof follows the strategy of the proof of Theorem 2 in Wilking \cite{Wil03}, utilizing the theory of error correcting codes applied to the isotropy representation at the fixed point. The proof  relies heavily on the  \hyperref[CL]{Connectedness Lemma} in  \cite{Wil03} and is by induction. 

Our next result is an extension of Theorem 2 in \cite{Wil03}, a half-maximal symmetry rank result for simply connected manifolds of dimension $n\geq 10$, which remains valid in all lower dimensions except 7 by Theorem A of Fang and Rong  \cite{FR03} for dimensions 8 and 9 and the \hyperref[MSR]{Maximal Symmetry Rank Theorem} for dimensions $\leq 6$. The result follows by combining  Theorems 2, 3, and 4 of \cite{Wil03} with Theorem A of \cite{FR03} and Theorems 3.4 and 3.5 in this article.

\begin{1/2MSR}  \label{1/2MSR}
Let $M^n$ be a closed, positively curved manifold admitting an isometric and effective $T^r$-action. If $n \neq 7$ and $r \geq {\frac{n}{4}} + 1$, then one of the following holds:
	\begin{enumerate}
	\item $M$ is homotopy equivalent to $S^n$, $\RP^n$, a lens space, or $\CP^{\frac n 2}$; or
	\item $M$ is homeomorphic to $\HP^{r-1}$ and $r = \frac{n}{4} + 1$.
	\end{enumerate}
\end{1/2MSR}

Unfortunately, we are unable to prove a proper $\Z_2$-analog of Theorem 2 in \cite{Wil03} that includes quaternionic projective space. However, we are able to prove the following result giving us information about the fixed-point sets of $\Z_2$-corank at most four.

\begin{main}\label{thm:nover4}
Let $M^n$ be a closed, positively curved manifold, and assume $\Z_2^r$ acts effectively by isometries on $M$ with fixed point $x$. Assume that $n \geq 15$ and that $$r \geq \frac{n+3}{4} + 1.$$ Then at least one of the following occurs:
	\begin{enumerate}
	\item For any subgroup of $\Z_2^r$ with corank at most four, the fixed-point set component at $x$, $F^m_x$, is homotopy equivalent to $S^m$, $\RP^m$, $\CP^{\frac m 2}$, or $S^m/\Z_k$ for some $k \geq 3$; or
	\item $M^n$ is a simply connected integer cohomology $\HP^{r-2}$ and $r=\tfrac{n}{4}+2$.
	\end{enumerate}
\end{main}

We remark that Conclusion (1) follows if $M$ itself is a sphere, a real or complex projective space, or a lens space, but we can only prove this weaker statement. Regarding Conclusion (2), note that the space $\HP^{r-2}$ has $T^{r-1}$ symmetry but not $T^r$ symmetry, not even topologically by Theorem 3 in \cite{Wil03}. However, we show in Section \ref{sec:Examples} that $\HP^{r-2}$ does admit a $\Z_2^{r}$-action with a fixed point, and moreover has the additional property that some corank four subgroup has a fixed-point set component $\HP^2$.

Observe that the symmetry assumption in Theorem \ref{thm:nover4} cannot be relaxed to $r \geq \tfrac{n+2}{4} + 1$, because there exists a positively curved metric on $\CP^{2r-3}/\Z_2$ with $\Z_2^r$ symmetry, with the property that $\Z_2^r$ fixes a point $x$ and there exists a $\CP^5/\Z_2$ containing $x$ that is the fixed-point set component of a $\Z_2$-corank four subgroup, see Example \ref{1/4}.   In this example, there is furthermore a $\Z_2^{r-3}$ whose fixed-point set component at $x$ is $\CP^3/\Z_2$, which is also not one of the spaces in Conclusion (1) of Theorem \ref{thm:nover4}. 
Similarly, we also show in Example \ref{1/4} that $S^{4r-5}/\Gamma$, where $\Gamma$ is  a finite, non-cyclic subgroup of $\Sp(1)$, has $\Z_2^r$ symmetry, a fixed point, and a subgroup $\Z_2^{r-2}$ with non-standard fixed-point set component $S^3/\Gamma$. We conjecture that these models might be the only ones that arise if the symmetry assumption in Theorem \ref{thm:nover4} is relaxed slightly, and we make partial progress toward a solution of the same in Theorem \ref{thm:nover4PLUS}.

Finally, we note that Conclusion (2) of Theorem \ref{thm:nover4} occurs only if $r = \tfrac{n}{4} + 2$. Moreover, in this case, the fixed-point set components of subgroups of corank at most three are as in Conclusion (1). In particular, one sees from Theorem \ref{thm:nover4} that every fixed-point set component $F^m$ of the $\Z_2^r$-action is a $\Z_2$-cohomology $S^m$, $\RP^m$, or $\CP^{\frac m 2}$. In the torus case and by switching to rational coefficients, a similar result was obtained by the first author in \cite{Ken13} under the assumption that $n \equiv 0 \bmod 4$ and that the rank $r$ of the torus is a logarithmic function of the dimension.  
More recently, the first author, Wiemeler, and Wilking in \cite{KWW} show that the assumption on $n$ is not necessary and that the assumption on $r$ can be relaxed to $r \geq 5$, see also Nienhaus \cite{N}. It would be interesting to explore the extent to which Theorem \ref{thm:nover4} can be improved, while noting that Theorem F of \cite{KWW} does not extend to $\Z_2$-torus representations.

As a bonus, Lemmas \ref{lem:Z2MSRlens} and \ref{lem:Z2MSRform} give us the following bounds on the maximal $\Z_2$-fixed point symmetry rank  for manifolds of positive curvature with non-trivial fundamental groups. Both estimates are sharp, see Examples \ref{ex1} and \ref{1/4}, respectively. 
\begin{main}\label{thm:pi1} Let $M$ be a closed, positively curved Riemannian manifold admitting an isometric and effective $\Z_2^r$-action 
with a fixed point. Then the following hold:
\begin{enumerate}
\item If $\pi_1(M)$ is neither trivial nor $\Z_2$, then $r\leq \tfrac{n+1}{2}$; and 
\item If $\pi_1(M)$ is a non-cyclic finite group, then $r\leq \tfrac{n+5}{4}$.
\end{enumerate}
\end{main}

Some remarks on the assumed existence of fixed points are in order. First, for closed manifolds of positive curvature, a result of Berger \cite{Ber} implies that every subgroup of the torus has non-empty fixed-point set in even dimensions, while a result of Sugahara \cite{Sug82} (see also \cite{GS}) shows that every torus action has a circle orbit in odd dimensions. Second, each component of the fixed-point set of a subtorus inherits a torus action from the ambient manifold since tori are connected, abelian groups. When passing from the torus case to the $\Z_2$-torus case, we lose both the existence of fixed-point set components and of an induced $\Z_2$-torus action.
 In fact, the theorem of \cite{Ber} does not strictly generalize to $\Z_2$-tori since, for example, even-dimensional spheres admit free and isometric $\Z_2$-actions. Likewise, the result of \cite{Sug82} does not generalize either, as both $(4k+3)$-dimensional real projective spaces in Example \ref{1/4}, and Eschenburg spaces,  see Shankar \cite{Sh}, admit free and isometric $\Z_2^2$-actions.  So we assume throughout that our $\Z_2$-tori have non-empty fixed-point sets. One advantage of this assumption is that it also implies that the $\Z_2$-torus acts invariantly on each of those fixed-point set components containing a point fixed by the whole $\Z_2$-torus. These observations are crucial to generalizing the arguments from the torus case.

One final remark illustrates why the proof techniques in the torus case with rank $r$ roughly correspond to those for a $\Z_2$-torus of rank $2r$. This doubling is based on the simple observation that the lowest dimensional irreducible representation of a circle is of dimension two, while for $\mathbb{Z}_2$, it is of dimension one. More specifically, the maximal symmetry rank bound is roughly doubled for the 
$\Z_2$-fixed point maximal symmetry rank. Likewise,  the induction setup and use of error correcting codes in \cite{Wil03} for a torus  of rank $r \approx \tfrac n 4$ are used in our proof of Theorem \ref{thm:nover2} for a $\Z_2$-torus of rank $r \approx \tfrac n 2$. Finally, the method of proof in work of Rong and Su \cite{RS05} for a torus of rank $r \approx \tfrac n 8$ motivates the proof of Theorem \ref{thm:nover4} with a $\Z_2$-torus of rank $r \approx \tfrac n 4$. 

\subsection*{Organization}

Section \ref{sec:ECC} contains the 
estimates from the theory of error correcting codes that will be used throughout most of the paper. Section \ref{sec:TransformationGroups} reviews results on $\Z_2$-torus actions on spaces with singly generated $\Z_2$-cohomology. Section \ref{sec:ConnectednessLemma} recalls the  Connectedness and Periodicity Lemmas in \cite{Wil03}, summarizes some of the applications of these ideas, and contains proofs of the corresponding refinements needed for the proofs of the main theorems. In Sections \ref{sec:n}, \ref{sec:nover2}, and \ref{sec:nover4}, we prove Theorems \ref{thm:n}, \ref{thm:nover2}, and \ref{thm:nover4}, respectively. Section \ref{sec:Examples} includes a summary, in order of decreasing symmetry, of the examples of positively curved manifolds with $\Z_2$-torus actions that illustrate these results.  Finally, Appendix \ref{sec:Appendix} contains a table of error correcting codes used in the proof of one of the results in Section \ref{sec:ECC} for easy reference.

\subsection*{Acknowledgements}

The authors are grateful to Tibor Macko for helpful conversations. They are also grateful to Austin Bosgraaf for a careful reading of the article. This material is based upon work supported by the National Science Foundation under Grant DMS-1440140 while the authors participated in a program hosted 
by the Mathematical Sciences Research Institute in Berkeley, California, during the summer of 2021.  L. Kennard was partially supported by NSF Grant DMS-2005280 and Simons Foundation Award MPTSM-00002791. C. Searle was partially supported by NSF Grant DMS-1906404 and DMS-2204324 and is grateful to Syracuse University for their hospitality during two visits while a part of this work was completed.

\smallskip\section{Estimates for error correcting codes}\label{sec:ECC}

In this section, we provide estimates derived from the theory of error correcting codes, as well as  applications of the same to the theory of positively curved manifolds. Before we begin, we recall the definition of the Hamming weight and the Hamming distance. When it is clear from context, we will refer to each of them simply as {\em weight} and {\em distance}, respectively.

\begin{definition}[{\bf Hamming weight}] Let $\iota\in \Z_2^n$ be an involution. The {\em Hamming weight} of $\iota$, denoted by $|\iota|$, is the number of non-trivial entries in $\iota$. The {\em Hamming distance} between two involutions $\iota_1, \iota_2\in \Z_2^n$ is $|\iota_1\iota_2|$. 
\end{definition}

In the theory of error correcting codes, one considers upper bounds on the minimum Hamming weight of the elements in the image of an injective, linear map $\Z_2^r \to \Z_2^n$. The optimal upper bound is denoted throughout this section by $d(n, r)$. 

Note that  in the case where $N$ is a fixed-point set component of an involution $\iota$ acting on a manifold $M$, where $\iota$ is regarded as a linear map on $T_xM$ for some $x \in N$, $\codim(N)$ equals the Hamming weight of the involution. 

The main result we use from the theory of error correcting codes is the following improvement of the Elias-Bassalygo bound,  which follows immediately from Proposition 3.1 in \cite{Wil03}, together with an argument in the proof of Corollary 3.2 in \cite{Wil03}.

\begin{theorem}\cite{Wil03}\label{thm:ECCWilking}
Fix $\delta \in (0,\tfrac 1 2)$. If $\rho:\Z_2^r \to \Z_2^n$ is a (not necessarily linear) injection such that $|\rho(\iota)| \geq \delta n$ for all non-trivial $\iota \in \Z_2^r$, then
	\[r \leq \of{1 - H(J(\delta))} n + 1.5 \log_2 n + 0.5\log_2\of{\frac{1-J(\delta)}{J(\delta)}} + 1.5,\]
where $J(\delta) = \tfrac 1 2\of{1 - \sqrt{1 - 2\delta}}$ and $H(\epsilon) = -\epsilon \log_2 \epsilon - (1-\epsilon)\log_2(1-\epsilon)$.
\end{theorem}

The function $J$ is related to the Johnson bound in the theory of error correcting codes, and $H$ is the entropy function. For our purposes, we only need the facts that $J(\delta)$ is an increasing function that maps the interval $[0,\tfrac 1 2]$ onto itself and that $1-H(\epsilon)$ is decreasing on the interval $0 \leq \epsilon \leq \tfrac 1 2$.

Here we sketch the arguments to avoid any confusion. To prove Theorem \ref{thm:ECCWilking},  we first need the following lemma which is contained in the proof of Proposition 3.1 in \cite{Wil03}.

\begin{lemma}\label{ECC}
Given $k$ vectors in the Hamming sphere $S_w(0)$ of radius $w$ about the origin in $\Z_2^n$ such that the distance between any two is at least $b$, then $k \leq n + 1$ provided $w < J\of{\tfrac b n} n$.
\end{lemma}

\begin{proof}[Sketch of the Proof of Lemma \ref{ECC}] The proof of Proposition 3.1 in \cite{Wil03} involves mapping the 
Hamming sphere $S_w(0) \subseteq \Z_2^n$ to a Euclidean $S_r(0) \subseteq \R^n$.
To do this, define $f:S_w(0) \to \R^n$ component-wise by $(x_i) \mapsto  \of{\tfrac{x_i}{w} - \tfrac{1-x_i}{n-w}}$. For $x \in S_w(0)$, the Hamming weight of $x$ is $w$, and the Euclidean norm of $f(x)$ is $r = \sqrt{\tfrac{n}{w(n-w)}}$. Similarly, given two vectors $x$ and $y$ among the $k$ vectors in the assumption, their Euclidean inner product satisfies
	\[\langle f(x), f(y) \rangle = 1 - \frac{r^2}{2} d(x,y) \leq 1 - \frac{r^2}{2} b,\]
where $d(x,y)$ denotes the Hamming distance and where $d(x,y) \geq b$ by assumption. One can verify that the right hand side is negative provided $w < J\of{\tfrac b n} n$, so the Euclidean angle between $x$ and $y$ is strictly larger than $\pi/2$. After scaling, we get $k$ vectors on the unit sphere in $\R^n$ such that the angle between any two is larger than $\pi/2$. By an elementary induction argument over $n$, one can show that this condition implies $k \leq n + 1$.
\end{proof}

\begin{proof}[Proof of Theorem \ref{thm:ECCWilking}]
Set $b = \delta n$ and $w_0 = \ceil{J(\delta) n} - 1$. Note that $d(x,y) \geq b$ for all pairs of distinct points in the image of $\rho$ and that $w_0$ is an integer less than $J(\delta) n$.

Set $C = \rho(\Z_2^r)$, and consider the set $\{(c,x) \in C \times \Z_2^n \st d(c,x) = w_0\}$. On the one hand, the number of elements in this set is $\sum_{c \in C} |S_{w_0}(c)| = 2^r \binom{n}{w_0}$. On the other hand, the number of elements is $\sum_{x \in \Z_2^n} |\{c \in C \st d(c,x) = w_0\}|$. But $\{c \in C \st d(c,x) = w_0\}$ maps bijectively to $\{c' \in x + C \st d(c',0) = w_0\}$ via $x \mapsto c+x$. Moreover, the latter set has the properties that all elements have weight $w_0$ and that any two have distance at least $b$ since translations in $\Z_2^n$ are isometries of the Hamming distance. Since $w_0 < J(\delta)n$, the lemma implies that each of these sets has at most $n + 1$ elements. Therefore we have the bound
	\[2^r \binom{n}{w_0} \leq 2^n (n+1).\]
Next, we set $w = w_0+1$ and rewrite this inequality as
	\[2^r \binom{n}{w} \leq 2^n\left(\frac{n-w+1}{w}\right)(n+1).\]
We now prove the result assuming $2 \leq w \leq \tfrac n 2$ and come back to the other cases at the end. Since $w \geq 2$, we may estimate the right-hand side to get
	\[2^r \binom{n}{w} \leq 2^n \frac{n^2}{w}.\]
On the left-hand side, we use the estimate
	\[\binom{n}{w} \geq 2^{H(\epsilon) n} / \sqrt{8(\epsilon)(1-\epsilon)n},\]
coming from Stirling's formula (see Lemma 4.7.1 in Ash \cite{Ash}), where $\epsilon = w/n$. Putting this together with the previous inequality, we have
	\[r \leq \of{1 - H\of{\frac w n}} n + 1.5 \log_2 n + \log_2\of{\frac{n-w}{w}} + 1.5.\]
Since $w \leq \tfrac n 2$ and $1 - H(\epsilon)$ is decreasing on the interval $[0,\tfrac 1 2]$, the entire right-hand side of this expression is decreasing in $w$.  Since $w \geq J(\delta) n$, we may substitute $J(\delta) n$ for $w$ in this bound to obtain the theorem. 

To finish the proof, we deal with the edge cases $w < 2$ and $w > \tfrac n 2$. The former implies that $w_0 = 0$ and hence that $r = 0$ because of the fact that the set $\{c' \in x + C \st d(c', 0) = w_0\}$ can only have one element of weight zero. The claimed upper bound on $r$ in the theorem is then easy to verify. The latter case implies that $w_0 \geq \tfrac{n-1}{2}$ and hence that $J(\tfrac b n) n > \tfrac{n-1}{2}$. Solving for $\delta = b/n$ in this expression, we obtain $\delta \geq \tfrac 1 2$. This contradicts the assumptions of the theorem, so the latter case does not occur.
\end{proof}

We now use Theorem \ref{thm:ECCWilking} together with some additional bounds from the theory of error correcting codes to prove the main lemma we need in our applications. We let
	\[B = \{4, 12\} \cup \{3, 7, 11, 23\}\]
denote the set of dimensions missing from Theorem \ref{thm:nover2}, and we define $\delta_B(n)$ to be one for $n \in B$ and zero otherwise. 

\begin{lemma}\label{lem:ECC}
Let $M^n$, $n\geq 3$, be an $n$-dimensional Riemannian manifold admitting an effective, isometric $\Z_2^r$-action 
with a fixed point $x$. Then the following hold:
       \begin{enumerate}[font=\normalfont]
       \item If $r\geq\tfrac{n+1}{2}+\delta_B(n)$, there is an involution $\iota\in\Z_2^r$ 
       such that $\codim(M_x^{\iota})\leq\tfrac{n+3}{4}$;
       \item If $r\geq \tfrac{n+7}{4}$ and if $n$ is odd, 
       then there exists an involution $\iota\in\Z_2^r$ 
       such that $\codim(M_x^{\iota})\leq \tfrac{n-1}{2}$.
              \end{enumerate}
       \end{lemma}
       
\begin{proof}
Consider the isotropy representation $\rho:\Z_2^r\to\Z_2^n$ at $x$. Recall that the Hamming weight of the image of an involution equals the codimension of its fixed-point set.

For Part (2), if all codimensions are at least $\tfrac{n+1}{2}$, then the Griesmer bound implies
       	\[n \geq \sum_{i=0}^{r-1} \ceil{\frac{n+1}{2^{i+1}}} \geq n+1 - \frac{n+1}{2^r}.\]
Since $r \geq \tfrac{n+7}{4}$, this implies
	\[\log_2(n + 1) \geq \frac{n+7}{4}.\]
This is a contradiction for all $n \geq 3$, so Part (2) follows.

Finally, we prove Part (1) using a combination of Theorem \ref{thm:ECCWilking}, an elementary estimate due to Hamming, and  
the use of code tables for small $n$. Assume that $\codim(M^\iota_x) > \tfrac{n+3}{4}$ for all non-trivial $\iota \in \Z_2^r$. 

First assume $n \geq 128$. By assumption, $|\rho(\iota)| \geq \tfrac 1 4 n$ for all non-trivial $\iota \in \Z_2^r$, so we may apply Theorem \ref{thm:ECCWilking} with $\delta = \tfrac 1 4$. 
The theorem implies that
	\[r \leq \of{1 - H(J(\delta))} n + 1.5 \log_2 n + \log_2\of{\frac{1 - J(\delta)}{J(\delta)}} + 1.5.\]
A calculation shows that $1 - H(J(\delta)) < 0.4$ and that the sum of the two constant terms is less than $2.772$. Combining this inequality with the estimate $r \geq \tfrac{n+1}{2}$, we have
	\[\frac{n+1}{2} \leq 0.4 n + 1.5 \log_2 n + 2.772.\]
One can see that this inequality fails for $n \geq 128$, so this completes the proof of Part (1) for values in this range.

Next, suppose that $68 < n < 128$. The elements in $\rho(\Z_2^r) \subseteq \Z_2^n$ have mutual distance at least $\tfrac{n+4}{4}$. In particular, the (closed) Hamming balls of radius $t = \floor{\tfrac 1 2\ceil{\tfrac n 4}}$ around each of these $2^r$ points are disjoint. Since the volume of each of these balls is larger than
	\[\binom{n}{t} + \binom{n}{t-1},\]
we obtain the estimate
	\[2^n > 2^r \of{\binom{n}{t} + \binom{n}{t-1}},\]
which is essentially the Hamming bound from the theory of error correcting codes. Now $r \geq \ceil{\tfrac{n+1}{2}}$, so we obtain
	\[n - \ceil{\frac{n+1}{2}} > \log_2\of{\binom{n}{t} + \binom{n}{t-1}}.\]
Calculating both sides for all values of $n$ in this range, we obtain a contradiction. This finishes the proof of Part (1) for all $n > 68$. 

Finally, for $n \leq 68$, we use the code tables included in Appendix \ref{sec:Appendix} to conclude that some involution exists with weight equal to $d(n,r)$ with $r = \ceil{\tfrac{n+1}{2}}$ and moreover that $d(n,r) \leq \tfrac{n+3}{4}$ for all values of $n$ in this range. This completes the proof of Part (1).
\end{proof}
      
       In the proof of the following lemma we use the technique of {\em shortening}, which we now describe.
      Let  $\rho:\Z_2^r\to\Z_2^n$ denote the isotropy representation of $\Z_2^r$ in $\Z_2^n$ and let $\iota_1\in \Z_2^r$.
Without loss of generality, we may assume that the first component of $\rho(\iota_1)$ is non-trivial. 
Let $\ker(\pi_1\circ \rho)=\Z_2^{r-1} \subseteq\Z_2^r$ denote the kernel of  $\pi_1\circ \rho$, where $\pi_1:\Z_2^n \to \Z_2$ is the  projection onto the first component. 
 Let $\bar{\rho}:\Z_2^{r-1} \to \Z_2^{n-1}$ be the composition of $\rho|_{\ker(\pi_{1}\circ \rho)}$ 
with the projection onto the last $n-1$ components. By our choice of subspace,  
we have $|\rho(\iota)|=|\bar{\rho}(\iota)|$ for $\iota\in\ker(\pi_{1}\circ \rho)$. In particular, it follows that $d(n, r)\leq d(n-1, r-1)$.

      \begin{lemma}\label{lem:ECCshorten} Let $\rho:\Z_2^r\hookrightarrow \Z_2^n$. Then the following hold:
      \begin{enumerate}
     
     \item If $r \geq 4$ and $n = 6$ or if $r \geq 8$ and $n \in \{14,15\}$,
      then for any non-trivial $\iota_1 \in \Z_2^r$, 
      there exists $\iota_2 \in \Z_2^r \setminus \langle\iota_1\rangle$ such that
      $|\rho(\iota_2)|$ is at most $2$ or $6$, respectively.

  \item  
  If $r \geq 5$ and if $n$ is $10$, $11$, or $12$, then
  for any non-trivial $\iota_1 \in \Z_2^r$,
  there exists $\iota_2 \in \Z_2^r \setminus \langle\iota_1\rangle$ such that 
  $|\rho(\iota_2)|$ is at most $4$, $4$, or $5$, respectively.
  
            \item        If $(r, n) = (5,11)$ and $\iota_1 \in \Z_2^5$ such that $|\rho(\iota_1)| = 4$, then
      	\begin{enumerate}
	\item There exists a non-trivial $\iota_2 \in \Z_2^5$ such that $|\rho(\iota_2)| \leq 3$; or
	\item There exist $\tau_1,\tau_2 \in \Z_2^5$ such that $|\rho(\tau_1)| = |\rho(\tau_2)| = 4$ and $|\rho(\tau_1\tau_2)| < 8$.
	\end{enumerate}
       \end{enumerate}
\end{lemma}

\begin{remark} In Part (3) it may not be possible  to choose  $\iota=\iota_1$, as the following example shows. For the map $\Z_2^5 \to \Z_2^{11}$ given by
	\[
	\begin{array}{rcl}
	\iota_1	&\mapsto&	(1,1,1,1,0,0,0,0,0,0,0)\\
	\iota_2	&\mapsto&	(1,1,0,0,1,1,1,1,1,1,1)\\
	\iota_3	&\mapsto&	(0,0,0,0,1,1,1,1,0,0,0)\\
	\iota_4	&\mapsto&	(0,0,0,0,1,1,0,0,1,1,0)\\
	\iota_5	&\mapsto&	(0,0,0,0,1,0,1,0,1,0,1)\\
	\end{array}
	\]
the image of $\iota_1$ has weight four, but there is no involution of weight three and there is no second involution of weight four whose product with $\iota_1$ has weight less than eight. Note however the conclusion still holds by choosing $\tau_1 = \iota_3$ and $\tau_2 = \iota_4$, for example.
\end{remark}

\begin{proof}[Proof of Lemma \ref{lem:ECCshorten}]

For Part (1), 
we apply the shortening technique to find an $\iota_2\not\in \langle\iota_1\rangle$ such that $|\rho(\iota_2)|$ is at most $d(5,3)$ when $n = 6$ or at most $d(n-1,7)$ when $n \in \{14,15\}$. A direct computation or an application of the Griesmer bound shows that $d(5,3) \leq 2$ and $d(13,7) \leq d(14,7) \leq 6$, as needed.

For Part (2), 
we apply the shortening technique to find an $\iota_2\not\in \langle\iota_1\rangle$ such that $|\rho(\iota_2)|\leq d(n-1, 4)$.
Again a direct computation or an application of the Griesmer bound shows that $d(11,4) \leq 5$, $d(10,4) \leq 4$, or $d(9,4) \leq 4$, as needed. 

We now prove Part (3). 
Assume $r=5$ and $n=11$. 
Using the shortening technique as above, 
choose a non-trivial $\iota_2\in \Z_2^5$ such that $|\rho(\iota_2)|\leq d(10, 4)=4$. 
If $|\rho(\iota_2)| \leq 3$, then Part ($3.a$) holds, 
and if $|\rho(\iota_2)| \leq 4$ and $|\rho(\iota_1\iota_2)| < 8$, then Part ($3.b$) holds for the pair $(\iota_1, \iota_2)$.  
Therefore we may assume that $|\rho(\iota_2)| = 4$ and $|\rho(\iota_1\iota_2)| = 8$. 
Note that the latter condition implies that $\rho(\iota_1)$ and $\rho(\iota_2)$ have no non-trivial entries in common.

We now perform the shortening technique to the first non-trivial entries of $\rho(\iota_1)$ and $\rho(\iota_2)$. 
That is, we let $\ker(\pi_{1,2}\circ \rho)=\Z_2^{3} \subseteq\Z_2^5$ denote 
the kernel of $\pi_{1, 2}\circ \rho$, where $\pi_{1, 2}:\Z_2^{11} \to \Z_2^2$ is the 
projection onto the first non-trivial components of both $\iota_1$ and $\iota_2$. 
Let $\bar{\rho}:\Z_2^{3} \to \Z_2^{9}$ be the composition of $\rho|_{\ker(\pi_{1,2}\circ \rho)}$ with 
the projection onto the remaining $9$ components. 
Then there exists an $\iota_3\in \ker(\pi_{1,2}\circ \rho)$ such that $|\rho(\iota_3)|\leq d(9, 3)$, which again can be shown easily to be at most four. If $|\rho(\iota_3)|\leq 3$, then Part ($3.a$) holds. Otherwise $\rho(\iota_3)$ has weight four and, since $n=11$, Part ($3.b$) holds for the pair $(\iota_1,\iota_3)$ or $(\iota_2,\iota_3)$. 

\end{proof}

\smallskip\section{Results from the theory of transformation groups}\label{sec:TransformationGroups}

The model spaces that motivate our theorems have 
singly generated cohomology with coefficients in $\Z_2$. By the work of Adams \cite{Ada60}, such spaces have the $\Z_2$-cohomology of 
a sphere, 
a real, complex, or quaternionic projective space, or 
the Cayley plane. 
Given a $\Z_2$-torus action on such a space, the fixed-point set components are also one of these spaces at the level of $\Z_2$-cohomology. Indeed, the classical result of Smith 
in the case of spheres states the following (see Smith \cite{Sm} or Theorem II.5.1 in Bredon \cite{Bre72}).

\begin{theorem}\cite{Sm}\label{thm:Z_p fixed points on spheres}
If $G$ is a $2$-group which acts on a smooth manifold $M^n$ with the $\Z_2$-cohomology of $S^n$,
then $M^G$ is a $\Z_2$-cohomology $r$-sphere for some $-1\leq r\leq n$.
\end{theorem}

Here $r = -1$ is simply convenient notation for the case where $G$ has no fixed points. For the other singly generated spaces, the proof in the rational case is due to Bredon. For the $\Z_2$-case, the result is the following (see Theorem VII.3.2 in \cite{Bre72}).

\begin{theorem}\cite{Bre72}\label{t:Z2fp}
Let $M$ be a closed manifold with the $\Z_2$-cohomology of ${\RP}^n$, ${\CP}^n$, or ${\HP}^n$ with $n \geq 2$.
If $\Z_2$ acts smoothly and effectively on $M$ with non-empty fixed-point set $F$, then one of the following occurs:
       \begin{enumerate}[font=\normalfont]
       \item $F$ is a $\Z_2$-cohomology $\RP^n$ and $M$ is a $\Z_2$-cohomology $\CP^n$.
              \item $F$ is a $\Z_2$-cohomology $\CP^n$ and $M$ is a $\Z_2$-cohomology $\HP^n$.
       \item $F$ has two components, $F_1$ and $F_2$, and if $M$ is a $\Z_2$-cohomology ${\FP}^{n}$, 
       then each $F_i$ is a $\Z_2$-cohomology ${\FP}^{n_i}$ such that $n=n_1+n_2+1$, where $n_i \geq 0$.
       \end{enumerate}
\end{theorem}

Theorem \ref{t:Z2fp}  allows us to compute the maximal $\Z_2$-symmetry rank of spheres and projective spaces. This result is probably well-known, but we remark that our proof uses induction over $n$ and Lemma \ref{lem:ECC}, our error correcting codes estimate.

\begin{theorem}\label{thm:curvature_free}
Let $\Z_2^r$ act smoothly and effectively on a smooth, closed manifold $M$ with a fixed point. Then the following hold:
       \begin{enumerate}[font=\normalfont]
       \item If $n \geq 0$ and $M$ is a $\Z_2$-cohomology ${\RP}^n$ or $S^n$, then $r\leq n$.
       \item If $n \geq 1$ and  $M$ is a $\Z_2$-cohomology ${\CP}^n$, then $r\leq n+1$.
       \item If $n \geq 2$ and $M$ is a $\Z_2$-cohomology ${\HP}^n$, then $r\leq n+2$.
       \end{enumerate}
\end{theorem}

\begin{proof}
Part (1) is immediate. Indeed, the isotropy representation at a fixed point $x$ induces an injective map 
$\Z_2^r\embedded \Z_2^n$ and hence the upper bound follows. 
\par
Now, we prove Part (2) by induction on $n$. The base case $n = 1$ holds by Part (1) since ${\CP}^1=S^2$. 
For the induction step, we choose a non-trivial $\iota\in\Z_2^r$ such that $N:=M_x^{\iota}$ is of maximal dimension. 
Since the action is effective, Theorem \ref{t:Z2fp} implies that $N$ has the $\Z_2$-cohomology 
of either ${\RP}^n$ or ${\CP}^k$ for some $k\leq n-1$. 
Note that since $\Z_2^r$ is abelian and $x$ is a fixed point of the action, $\Z_2^r$ acts on $N$. 
By maximality of $N$, the kernel of this action is $\Z_2$ and hence $N$ admits an effective $\Z_2^{r-1}$-action
with non-empty fixed-point set. Now, if $N$ has the $\Z_2$-cohomology of ${\RP}^n$,
then the result follows by Part (1). If instead $N$ has the $\Z_2$-cohomology of ${\CP}^k$ 
for some $k \leq n-1$, then the induction hypothesis implies that $r\leq n+1$, as claimed.
\par
Next, we prove Part (3) by induction on $n$. For the base case, it suffices to show that $\Z_2^5$ cannot act effectively 
and with non-empty fixed-point set on a $\Z_2$-cohomology ${\HP}^2$. 
Suppose, to obtain a contradiction, that there exists such a $\Z_2^5$-action on a manifold $M$ with the $\Z_2$-cohomology 
of ${\HP}^2$. By Part (1) of Lemma \ref{lem:ECC}, there exists non-trivial $\iota\in\Z_2^5$
such that $N:=M_x^{\iota}$ has codimension at most two. But $N$ has the $\Z_2$-cohomology of ${\HP}^0$, ${\HP}^1$, 
or ${\CP}^2$ by Theorem \ref{t:Z2fp}, giving us the desired contradiction.
\par
For the induction step, we argue exactly as in the proof of Part (2). Again, choose non-trivial $\iota\in\Z_2^r$ 
such that $N:=M_x^{\iota}$ is of maximal dimension. By Theorem \ref{t:Z2fp}, it follows that $N$ has the $\Z_2$-cohomology 
of either ${\CP}^n$ or ${\HP}^k$ for some $k\leq n - 1$. 
Hence it follows either by the induction hypothesis or by Part (2), respectively, that $r\leq n + 2$, as claimed.
\end{proof}

\smallskip\section{Applications of the Connectedness Lemma}
\label{sec:ConnectednessLemma}

In this section, we recall Wilking's Connectedness and Periodicity Lemmas. Together these imply a strong cohomological condition on totally geodesic submanifolds in positively curved manifolds. We then generalize  a result in \cite{Ken13} that allows us to lift this periodicity to the ambient manifold. Finally we pull together these results and other previous work to prove, in some cases, that the existence of a totally geodesic submanifold of codimension at most four implies a computation of the homotopy type of the ambient manifold.

We begin by recalling the definition of $k$-periodicity.
\begin{definition}[{\bf $k$-periodicity}]  Let $M^n$ be a connected, $n$-dimensional manifold.
We say that $x$ induces {\em $k$-periodicity}, or $M^n$ has {\em $k$-periodic cohomology} provided there exists $x\in H^k(M^n;\Z)$ such that the homomorphism $H^i(M^n;\Z)\overset{\cup x}{\rightarrow}H^{i+k}(M^n;\Z)$
is surjective for $0\leq i<n-k$ and injective for $0<i\leq n-k$.
\end{definition}

\subsection{Wilking's Connectedness and Periodicity Lemmas}

We recall the Connectedness Lemma of \cite{Wil03}.
\begin{CL}\label{CL} \cite{Wil03}
Let $M^n$ be a closed Riemannian manifold with positive sectional curvature. The following hold:
       \begin{enumerate}[font=\normalfont]
       \item If $N^{n-k}$ is a closed, totally geodesic submanifold of $M$, then the inclusion map $N^{n-k}\embedded M^n$ 
       is $(n-2k+1)$-connected.
       \item If $N_1^{n-k_1}$ and $N_2^{n-k_2}$ are closed, totally geodesic submanifolds of $M$ with $k_1 \leq k_2$ and $k_1 +k_2\leq n$, then 
       the inclusion $N_1^{n-k_1}\cap N_2^{n-k_2}\embedded N_2^{n-k_2}$ is $(n-k_1-k_2)$-connected.
       \end{enumerate}
\end{CL}

We also recall Lemma 2.2 of \cite{Wil03}.
\begin{PL}\cite{Wil03}
\label{PL}
If an inclusion $N^{n-k}\embedded M^n$ of closed, orientable manifolds is $(n-k-\ell)$-connected with $n - k - 2\ell > 0$, then there exists $e\in H^k(M;\Z)$ such that the homomorphism $H^i(M;\Z)\overset{\cup e}{\rightarrow}H^{i+k}(M;\Z)$ is a surjection for $i = \ell$, an injection for $i + k = n - \ell$, and an isomorphism everywhere in between.
\end{PL}

\subsection{Lifting periodicity}

In a typical application of the \hyperref[CL]{Connectedness Lemma} and the \hyperref[PL]{Periodicity Lemma}, one has a totally geodesic submanifold with periodic cohomology, and one wants to lift this periodicity to the ambient manifold. One instance of this is the following direct consequence of Theorem 1 of Wilking \cite{Wil03}. 

\begin{theorem}\label{thm:purple-OneSubmanifold}\cite{Wil03}
Let $N^{n-k} \subseteq M^n$ be a totally geodesic embedding of closed, simply connected, positively curved manifolds such that $k \leq \tfrac{n+3}{4}$. If $N$ has the cohomology of a sphere, complex projective space, or quaternionic projective space, then the same holds for $M$. Moreover, $k$ is even in the case of complex projective spaces, and $k$ is divisible by four in the case of quaternionic projective spaces.
\end{theorem}

A variant of the above situation that arises frequently involves the case of two closed, totally geodesic submanifolds, $N_1$ and $N_2$, one of which  
has periodic cohomology. In order to conclude that $M^n$ also has periodic cohomology, one requires further use of the \hyperref[CL]{Connectedness Lemma}. If the codimensions satisfy $2k_1 + 2k_2 \leq n$, this was accomplished in \cite{Ken13}. The following proposition shows that one can relax this bound somewhat. Even though the improvement is minimal, we need this version later.

\begin{proposition}\label{pro:purple}
Let $M^n$ be a closed, simply connected, positively curved Riemannian manifold. Suppose that $N_1^{n-k_1}$ and $N_2^{n-k_2}$ are two closed, totally geodesic submanifolds of $M$ with $k_1\leq k_2$. Let $R$ be $\Z$ or $\Z_p$, for some prime $p$. The following hold:
       \begin{enumerate}[font=\normalfont]
       \item Suppose that $4k_1\leq n+3$ and $k_1+2k_2\leq n+1$. If $N_2$ is an $R$-cohomology sphere, then so is $M$.
       \item Suppose that $4k_1\leq n+3$ and $g + k_1 + 2k_2 \leq n + 3$ for some $g \geq 2$ that divides $k_1$. If $N_2$ has $g$-periodic $R$-cohomology, then so does $M$.
       \end{enumerate}
\end{proposition}

We remark that $1$-periodic cohomology is equivalent to being a homology sphere for simply connected manifolds. 
Therefore Part (1) may be viewed as the $g = 1$ case of Part (2), with a slight adjustment on the assumptions
 on the codimensions.  In Part (2),  the case where $g=2$ follows directly from Lemma 7.2 of \cite{Wil03}.
 We further note that, with a modification of the proof, one can remove the condition that $g$ divides $k_1$ in Part (2). 

\begin{proof}[Proof of Proposition \ref{pro:purple}]
We begin with the proof of Part {\rm(1)}. 
Note that $M$ is a $\Z$-cohomology sphere if and only if it is a $\Z_p$-cohomology sphere for all primes $p$, by the universal coefficients theorem. Therefore, the result for $R = \Z$ follows from the result for $R = \Z_p$ for all $p$. Therefore we may assume $R = \Z_p$ for an arbitrary fixed prime $p$.
\par
Note that by the \hyperref[CL]{Connectedness Lemma}, the inclusion $N_2\embedded M$ is $(n-2k_2+1)$-connected 
and hence at least $k_1$-connected by the assumption $k_1 + 2k_2 \leq n + 1$. This, together with the assumption 
that $N_2$ is a $\Z_p$-cohomology sphere, implies that $H^i(M;\Z_p)=0$ for $1\leq i\leq k_1$. 

Now, choose $e_1\in H^{k_1}(M;\Z_p)$ as in the \hyperref[PL]{Periodicity Lemma} applied to the inclusion 
$N_1\embedded M$. In particular, multiplication by $e_1$ induces injections $H^i(M;\Z_p) \to H^{i+k_1}(M;\Z_p)$ 
for $k_1-1< i\leq n - 2k_1 + 1$. Now $H^{k_1}(M;\Z_p) = 0$, so $e_1 = 0$. Hence the groups $H^i(M;\Z_p) = 0$ 
for $1\leq i\leq n-2k_1+1$. By the assumption $4k_1\leq n+3$, it follows that $H^i(M;\Z_p) = 0$ for all 
$1\leq i\leq \floor{\tfrac n 2}$. By Poincar\'e Duality, $M$ is a $\Z_p$-homology sphere. 
This completes the proof of Part ($1$).

We now prove Part {\rm(2)}.  As mentioned above, the case where $g=2$ follows directly from Lemma 7.2 of \cite{Wil03}, so we assume $g\geq 3$. Let $R$ be $\Z$ or $\Z_p$ for some prime $p$. We abbreviate $H^*(M;R)$ by $H^*(M)$. 
We claim that the map $\iota^*:H^{k_1}(M)\to H^{k_1}(N_2)$ induced by inclusion is an isomorphism. 
In order to prove this, note that $\iota^*$ is an isomorphism in degree $i$ for $i\leq n-2k_2$ by the \hyperref[CL]{Connectedness Lemma}. 
Since $g\geq 3$, we have $k_1\leq n-2k_2$ and hence the claim follows. 

Choose $x_2\in H^g(N_2)$ inducing $g$-periodicity, 
 and let $x_1\in H^g(M)$ be a pre-image of $x_2$ under the map $\iota^*:H^g(M) \to H^g(N_2)$.  
 Next, note that $e_2 = x_2^{k_1/g} \in H^{k_1}(N_2)$ also induces periodicity.
Since $\iota^*$ is an isomorphism in degree $k_1$, we can choose a pre-image $e_1\in H^{k_1}(M)$ of $e_2$ under $\iota^*$, and we have that $e_1 = x_1^{k_1/g}$.

To finish the proof, we claim that $x_1$ induces periodicity in $H^*(M)$. That is, we must show that the map 
$H^i(M) \to H^{i+g}(M)$ induced by multiplication by $x_1$ is surjective for $0\leq i<n-g$ and injective for $0<i\leq n-g$. 
We break the proof into cases. Note that for the two sets of $i$-values that appear in both Cases $1$ and $2$ and in both Cases $2$ and $3$, respectively,  
we show injectivity in the first of these cases 
and surjectivity in the second of these cases.

\vspace{0.1cm}\noindent{\bf Case 1: \boldmath{$0 \leq i \leq k_1-1$}.}

By the naturality of cup products, we have the following commutative diagram:\\
\hspace*{5cm}\begin{tikzcd}
H^{i}(M)  \arrow{r}{\cup x_1} \arrow{d}{\iota^*}
                  & H^{i+g}(M) \arrow{d}{\iota^*}\\
H^i(N_2) \arrow{r}{\cup x_2}
                  & H^{i+g}(N_2).
\end{tikzcd}\\
For $0\leq i<k_1-1$, periodicity in $H^*(N_2)$ implies that multiplication by $x_2$ is surjective. 
In addition, the assumption $g+k_1+2k_2\leq n+3$ and the  \hyperref[CL]{Connectedness Lemma} imply that the first vertical map 
is an isomorphism and that the second vertical map is injective. Therefore multiplication by $x_1$ is surjective in this range.

For $0< i\leq k_1-1$, we have similarly that multiplication by $x_2$ is an isomorphism and that the vertical map on the left 
is an isomorphism, so multiplication by $x_1$ is injective in this range.

\vspace{0.2cm}\noindent{\bf Case 2: \boldmath{$k_1-1 \leq i \leq n-k_1-g+1$}.}

Since $e_1$ generates $H^{k_1}(M)$, it satisfies the conclusion of the \hyperref[PL]{Periodicity Lemma} applied to the inclusion 
$N_1\embedded M$. Now consider the maps
	\[H^i(M) \stackrel{\cup x_1}{\longrightarrow} H^{i+g}(M) \stackrel{\cup x_1^{\frac{k_1}{g}-1}}{\longrightarrow} H^{i+k_1}(M) \stackrel{\cup x_1}{\longrightarrow} H^{i+k_1+g}(M).\]
The composition $f_1$ of the first two maps is the same as multiplication by $e_1$ as is the composition $f_2$ 
of the last two maps.

First, for $k_1-1\leq i <n-2k_1+2-g$, $f_1$, and hence multiplication by the middle map, is surjective. 
In addition, $f_2$ is injective on this range, so in fact the middle map is an isomorphism. It now follows that the first map 
is surjective on this range.

Second, for $2k_1- g-1\leq i<n-k_1-g+1$, we only have to focus on the composition of the last two maps 
and shift indices down by $g$. Since multiplication by $e_1$ is surjective for this range of values of $i$, so is the second map in this composition. 

So far, we have shown that multiplication by $x_1$ is surjective for $k_1-1\leq i<n-k_1-g+1$. 
Indeed the assumption $4k_1\leq n+3$ implies that we have covered the full range of indices. 
The proof of injectivity is analogous, based on the subcases $k_1-1< i\leq n-2k_1+1$ and $i\leq n-k_1-g+1$, so we omit it.

\vspace{0.2cm}\noindent{\bf Case 3: \boldmath{$n-k_1- g+1\leq i\leq n-g$}.}

Consider the following commutative diagram:\\
\hspace*{1.5cm}\begin{tikzcd}
H^i(M) \arrow{r}{\mathrm{P.D.}} \arrow{d}{\cup x_1}
                                & H_{n-i}(M) \arrow{d}{\cap x_1}
                                & H_{n-i}(N_2) \arrow{l}{\iota_*}\arrow{d}{\cap x_2}
                                & H^{i-k_2}(N_2) \arrow{l}{\mathrm{P.D.}}\arrow{d}{\cup x_2}\\
H^{i+k_1}(M) \arrow{r}{\mathrm{P.D.}}
                 & H_{n-i-k_1}(M)
                 & H_{n-i-k_1}(N_2)\arrow{l}{\iota_*}
                 & H^{i+k_1-k_2}(N_2) \arrow{l}{\mathrm{P.D.}}
\end{tikzcd}\\

By our choice of $x_2$, the map on the right-hand side, given by multiplication by $x_2$, induces a surjection for $n-k_1-g+1\leq i<n-g$ and an injection for $n- k_1-g+1<i\leq n-g$. Note in addition that the Poincar\'e Duality maps 
are isomorphisms and that the bottom map induced by inclusion is an isomorphism in this range by the \hyperref[CL]{Connectedness Lemma}.

For $n-k_1-g+1\leq i<n-g$, this is enough to see that the surjectivity of multiplication by $e_2$ implies that multiplication 
by $e_1$ is surjective.

For $n- k_1-g+1<i\leq n-g$, we note in addition that the assumption $g+k_1+2k_2\leq n+3$ implies that the top map 
induced by inclusion is surjective by the \hyperref[CL]{Connectedness Lemma}. It follows that multiplication by $x_1$ is injective in this range.

By putting together the conclusions in Cases 1--3, it follows that multiplication by $x_1$ induces a map 
$H^i(M)\to H^{i+g}(M)$ that is surjective for $0\leq i< n-g$ and injective for $0<i \leq n-g$. 
Therefore $M$ has $g$-periodic integer cohomology, as claimed.
\end{proof}

\subsection{From periodic cohomology to homotopy type}

In this section, we demonstrate how to transition from $1$-, $2$-, $3$-, or $4$-periodic cohomology to a computation of the homotopy type. This requires additional assumptions and, we note, cannot always be accomplished. In particular, knowing that a four-periodic cohomology ring looks like that of $\HP^k$ is, in general, not enough to allow us to conclude rigidity up to homotopy equivalence.

The proof strategy is always to compute the fundamental group and the cohomology of the universal cover. When these invariants agree with a sphere, real projective space, lens space, or complex projective space, we automatically get homotopy rigidity by applying the following theorem, which combines some results in the existing literature.

\begin{theorem}\label{thm:cohomology-to-homotopy}
Let $M^n$ be a closed, smooth manifold, and let $\widetilde M^n$ denote the universal cover. The following hold:
	\begin{enumerate}
	\item If $\pi_1(M)$ is cyclic and $\widetilde M$ is a cohomology sphere, then $M$ is homotopy equivalent to $S^n$, $\RP^n$, or $S^n/\Z_l$ for some $l \geq 3$. In the last case, $n$ is odd; and 
	\item If $M$ is simply connected and has the cohomology of $\CP^{\frac n 2}$, then $M$ is homotopy equivalent to $\CP^{\frac n 2}$.
	\end{enumerate}
\end{theorem}

\begin{proof}
Statement (1) in the simply connected case follows from the classical theorems of Hurewicz and Whitehead. The case where $\pi_1(M) \cong \Z_2$ can be reduced to the sphere case by a geometric argument of Hirsch and Milnor in \cite{HirMil}. The case with $|\pi_1(M)| \geq 3$ is much more difficult and follows from results in non-simply connected surgery theory. First note that $n = 2m-1$ in this case since homology spheres in even degrees admit no free action by groups of order larger than two. Second, in Chapter 14.E of Wall's book \cite{Wa}, it is shown that the homotopy classes of closed $(2m-1)$-manifolds with fundamental group $\Z_N$ and universal cover homotopy equivalent to $S^{2m-1}$ are in one-to-one correspondence with the units in $\Z_N$. Third, it is shown in \cite[Chapter 14.E]{Wa} that, for any unit $d \in \Z_N$, the classical lens space $S^{2n-1}/\Z_N$ defined by the the action
	\[\omega \cdot(z_1,z_2,\ldots,z_n) = \of{\omega^d z_1, \omega z_2, \ldots, \omega z_n}\]
corresponds to the unit $d \in \Z_N$ under this correspondence. Combining the last two observations, it follows that $M$ is homotopy equivalent to a (classical) lens space.

Statement (2) is also well known. Using the Brown representability theorem, we may choose a map $f:M^n \to \CP^\infty$ that pulls back the generator in $H^2(\CP^\infty;\Z)$ to the generator of $H^2(M;\Z)$. Using the CW approximation theorem, we may homotope $f$ to a map into $\CP^{\frac n 2}$ in such a way that the pullback also satisfies this property on the second cohomology group. But now $f$ induces an isomorphism on cohomology as a result of the cup product structure, an isomorphism on homology by the universal coefficients theorem, and finally on homotopy by Hurewicz's theorem. Whitehead's theorem now implies that $f$ is a homotopy equivalence.
\end{proof}

As alluded to earlier, we remark that we do not know if similar cohomology-to-homotopy results hold for quaternionic projective space, or for $\CP^{2m+1}/\Z_2$ or spherical space forms with non-cyclic fundamental groups. This uncertainty is the reason why Conclusion (2) in Theorem \ref{thm:nover4} does not include a claim on the homotopy type (see also Theorem \ref{thm:nover4PLUS}).

Finally, we prove an analog of Theorem 3 of \cite{Wil03}.
Recall that in Theorem 3 of \cite{Wil03}, it is shown that a smooth torus action on a $(4m)$-dimensional  integral cohomology quaternionic space has rank at most $m+1$ and in the equality case, the manifold is homeomorphic to $\HP^m$. Here we prove an analog for $\CP^{2n+1}/\Z_2$.

\begin{theorem}\label{smoothbound} Let $M^{4m+2}$ be a closed manifold whose universal cover is a rational cohomology $\CP^{2m+1}$ and whose fundamental group is $\Z_2$. If $T^r$ acts smoothly and effectively on $M$, then $r\leq m+1$.
\end{theorem}

\begin{proof}
We lift the torus action to the universal cover $\widetilde M$, and we note that the $\Z_2$-action by deck transformations commutes with it. To prove the theorem, we actually prove the following statement: A closed manifold $P$ that is a rational cohomology $\CP^{2m+1}$ and admits commuting actions by a torus $T^r$ and by $\Z_2$, where the $\Z_2$ acts freely, satisfies $r \leq m + 1$.

First, if $m = 0$, then $P$ has dimension two and is a rational cohomology $\CP^1 = S^2$. Hence $P$ is diffeomorphic to $S^2$ and $r \leq 1$.

We may now assume $m \geq 1$ and proceed by induction. Choose $S^1 \subseteq T^r$ and a fixed-point set component $Q \subseteq P$ of maximal dimension. Note that there are fixed points since $\chi(P) = \chi(\CP^{2m+1}) > 0$, and note that $Q$ is a rational cohomology $\CP^k$ for some $k < 2m+1$ by Bredon's theorem. In addition, $Q$ admits a smooth, effective $T^{r-1}$-action by maximality, and this action has a fixed point. In particular, $r - 1 \leq k$. If $k \leq m$, then this already implies the claimed upper bound, so we may assume that $k \geq m + 1$. By Bredon's theorem, $Q$ is the unique fixed-point set component of $S^1$ with $\dim(Q)> \dim(P)/2$. In particular, the induced, free $\Z_2$-action on the fixed-point set of $S^1$ restricts to a free $\Z_2$-action on $Q$. Hence $Q$ has even Euler characteristic, $k$ is odd, and hence $k =2l+1 \leq 2m-1$. By the induction hypothesis, $r - 1 \leq l + 1$, so we conclude $r \leq l + 2 \leq m + 1$, as desired.
\end{proof}

\subsection{The Borel Formula} 
 We first recall the Borel Formula which plays an essential role in the proof of Theorem \ref{thm:nover4PLUS} and hence in the proof of Theorem \ref{thm:nover4}.

\begin{Borel}\label{Borel}\cite{Bor}
Let $\Z_2^r$ act smoothly on $M$, a Poincar\'e duality space, with fixed-point set component $F$. Then
	\[\codim(F \subseteq M) = \sum \codim(F \subseteq F')\]
where the sum runs over fixed-point set components $F'$ of corank one subgroups $\Z_2^{r-1} \subseteq \Z_2^r$ for which $\dim(F') > \dim(F)$.

These subgroups are the kernels of the irreducible subrepresentations. In particular, the number of pairwise inequivalent irreducible subrepresentations of the isotropy representation at a normal space to $F$ is at least $r$ if the action is effective. Moreover, equality holds only if the isotropy representation is equivalent to one of the form
	$$(\ep_1,\ldots,\ep_r) \mapsto \diag(\ep_1 I_{m_1}, \ldots, \ep_r I_{m_r})$$
where $\ep_i \in \{\pm 1\}$, $m_i > 0$ denote the multiplicities, $I_m$ denotes the identity matrix of rank $m$.
\end{Borel}

The simplicity of the representation when it has a minimal number of irreducible subrepresentations leads to the following observation. 
\begin{observe}\label{borel} Let $F_j^{m_j}$ be the fixed-point set of a 
$\Z_2^{r-j}$-action by isometries on a closed, positively curved manifold, and suppose that $r-j\geq 2$. If the isotropy
representation has exactly $r-j$ irreducible subrepresentations, and if the fixed-point set component $F_{j+1}$ containing $F_j$ of some $\Z_2^{r-j-1}$ has the property that the codimension $k_j$ of the inclusion $F_j \subseteq F_{j+1}$ is minimal, then this inclusion is $\dim(F_j)$-connected.
\end{observe}

\begin{proof}[Proof of Observation \ref{borel}]
By the \hyperref[Borel]{Borel Formula},  $F_j$ equals the intersection of $r-j$ pairwise transverse
 submanifolds, which are $\mathrm{Fix}(\iota_1), \hdots, \mathrm{Fix}(\iota_{r-j})$.
In turn, this means $F_j$ is the transverse intersection of $F_{j+1}$ and
some other fixed-point set component of another copy of $\Z_2^{r-(j+1)}$, say $F_{j+1}'$, inside the fixed-point set component of some $\Z_2^{r-j-2}$.
Since $\codim(F_j \subseteq F_{j+1})$ is minimal, Part 2 of the
\hyperref[CL]{Connectedness Lemma} implies that
$F_j \subseteq F_{j+1}$ is $\dim(F_j)$-connected.

\end{proof}

\subsection{The Small Codimension Lemmas}
We proceed now to the lemmas we use directly later in the paper. We note that many of these results are already known in the simply-connected case by work in Section 7 of \cite{Wil03}. In particular, the first is well-known and the simply connected case  follows from Lemma 7.1 in \cite{Wil03}.

\begin{Codim1}\phantomsection\label{Codim1} Let $M^{m+1}$ be a closed, positively curved Riemannian manifold. Suppose $N^m$ is a closed, totally geodesic submanifold of $M$, then $N^m$ is homotopy equivalent to $S^m$ or $\RP^m$.
\end{Codim1}

Note that in the special case where $N$ arises as the fixed-point set component of an isometric $\Z_2$-action,  the stronger conclusion that $N$ and $M$ are diffeomorphic to one of these spaces follows from Theorem \ref{thm:FangGrove}. Here we only need the \hyperref[CL]{Connectedness Lemma} and the \hyperref[PL]{Periodicity Lemma}, and include the proof for the sake of completeness.

\begin{proof}[Proof of the Codimension One Lemma]
We lift to the universal cover $\widetilde M$, and let $\pi^{-1}(N) \subseteq \widetilde M$ denote the preimage of $N$. By the \hyperref[CL]{Connectedness Lemma}, $\pi^{-1}(N)$ is simply connected and hence the universal cover of $N$, so  $\pi^{-1}(N)=\widetilde N$. In addition, the fact that $N \subseteq M$ is $m$-connected implies that $\widetilde N \subseteq \widetilde M$ is $m$-connected since covering projections induce isomorphisms on higher homotopy groups.

The \hyperref[PL]{Periodicity Lemma} implies that multiplication by some $e \in H^1(\widetilde M;\Z)$ 
induces maps $H^i(\widetilde M;\Z) \to H^{i+1}(\widetilde M;\Z)$ that are surjections for $0 \leq i < m$ and injections for degrees 
$0 < i \leq m$. Since $\widetilde M$ is simply connected, it follows that $\widetilde M$ is a cohomology sphere. The same holds for $\widetilde N$ by the \hyperref[CL]{Connectedness Lemma}. Finally since $\pi_1(N) \cong \pi_1(M)$ by the  \hyperref[CL]{Connectedness Lemma}, 
and since either $N$ or $M$ has even dimension, these fundamental groups are either trivial or $\Z_2$. 
It follows that $N$ is homotopy equivalent to $S^m$ or $\RP^m$ by Part (1) of Theorem \ref{thm:cohomology-to-homotopy}.
\end{proof}

The following lemma is similar to the \hyperref[Codim1]{Codimension One Lemma}. Part (1) follows from Corollaries 1.7 and 1.9 in Frank, Rong, and Wang \cite{FRW13}. Moreover, in the simply-connected case, the result follows from  Lemma 7.1 and Parts (1) and (2) of  Proposition 7.3 in \cite{Wil03}. 

\begin{Codim2}\phantomsection\label{codim2} Let $M^{m+2}$ be a closed, positively curved Riemannian manifold. Suppose $N^m$ is a closed, totally geodesic submanifold of $M$.
Assume one of the following holds:
	\begin{enumerate}[font=\normalfont]
	\item $M$ is odd-dimensional;
	\item The inclusion $N^m \subseteq M^{m+2}$ is $m$-connected; or
	\item  $M^{m+2}$ admits a $\Z_2^2$ action such that $N^m$ is a fixed-point set component of one involution and such that there is a second involution whose fixed-point set component $N'$ has codimension  at most $\tfrac{m+1}{2}$;
	\end{enumerate}
then $N^m$ is homotopy equivalent to $S^m$, $\RP^m$,  $\CP^{\frac m 2}$, or a lens space.
\end{Codim2}

\begin{proof}

First, suppose that $M$ satisfies Part (1). Then  the result follows by Corollaries 1.7 and 1.9 in Frank, Rong, and Wang \cite{FRW13}. 

As in the proof of the \hyperref[Codim1]{Codimension One Lemma}, we lift to the universal covers $\widetilde M$ and $\widetilde N$. 
We now suppose that $N$ satisfies Part (2). This means that the inclusion $\widetilde N \subseteq \widetilde M$ is also $m$-connected and hence $H^*(\widetilde M;\Z)$ 
is two-periodic by the \hyperref[PL]{Periodicity Lemma}. Since $\widetilde M$ is simply connected, it follows from the definition of two-periodicity that $\widetilde M$ is a cohomology sphere or $\CP^{\frac m 2 + 1}$. Correspondingly, $\widetilde N$ is a cohomology sphere or $\CP^{\frac m 2}$ by the \hyperref[CL]{Connectedness Lemma}. In the former case, $\pi_1(N) \cong \pi_1(M)$ is cyclic by Synge's theorem and Corollary 1.9 in \cite{FRW13}, so $N$ is homotopy equivalent to $S^m$, $\RP^m$ or to $S^m/\Z_k$ for some $k \geq 3$. 
In the latter case, either $\widetilde N$ or $\widetilde M$ has odd Euler characteristic and hence does not admit
a free $\Z_2$-action. Hence $\pi_1(N) \cong \pi_1(M)$ is trivial and $N \simeq \CP^{\frac m 2}$. 
This completes the proof of Part (2).

We now prove the theorem assuming Part (3). Let $\iota_1$ denote the involution fixing $N^m$, and let $\iota_2$ 
be any other non-trivial involution in $\Z_2^2$. By replacing $\iota_2$ by $\iota_1\iota_2$, if necessary, 
we may assume that the fixed-point set component $N'$ of $\iota_2$ either intersects transversely with $N$ 
or has the property that the codimension of $N\cap N'$ in $N'$ is one. 

In the former case, the \hyperref[CL]{Connectedness Lemma}, the \hyperref[PL]{Periodicity Lemma}, and Part (2)  imply that 
$N'$ is homotopy equivalent to $S^m/\Z_l$ for some $l \geq 1$ or to $\CP^{\frac m 2}$. In the latter case, $N'$ 
is homotopy equivalent to a sphere or to a real projective space by the \hyperref[Codim1]{Codimension One Lemma}. The same conclusions hold for $M$ and $N$ by Proposition \ref{pro:purple} and the \hyperref[CL]{Connectedness Lemma}, respectively, so the proof is complete.
\end{proof}

Note that with only the hypothesis that $N$ is a closed, totally geodesic submanifold of $M$, a closed,  positively curved Riemannian manifold, the \hyperref[CL]{Connectedness Lemma} only tells us that $N^m \subseteq M^{m+2}$ is $(m-1)$-connected. When $N^m$ is additionally a fixed-point set component of an isometric circle action, then the inclusion is $m$-connected and  Part (2) of the \hyperref[codim2]{Codimension Two Lemma} holds. Alternatively, in this case, one may apply the classification result of Grove and Searle \cite{GS} 
(see also Lemma 4.1 in \cite{AK20}). 

To prove our main theorems, we also require results involving both codimension three and four totally geodesic inclusions. 

\begin{Codim3}\phantomsection\label{codim3} Let $M^{m+3}$ be a closed, positively curved Riemannian manifold. Suppose $N^m$ is a closed, totally geodesic submanifold of $M$.
Suppose that: 
	\begin{enumerate}
	\item For $m\geq 3$, the inclusion $N^m \subseteq M^{m+3}$ is $m$-connected; or
	\item For $m\geq 6$, $M^{m+3}$ admits a $\Z_2^2$ action such that $N^m$ is a fixed-point set component of one involution and such that there is a second involution whose fixed-point set component $N'$ has codimension  at most $\tfrac{m+1}{2}$;
	
	\end{enumerate}
then, $N^m$ is homotopy equivalent to $S^m$ or $\RP^m$.
\end{Codim3}

\begin{proof}
Note that in Parts (1) and (2),  $\pi_1(N) \cong \pi_1(M)$ by the \hyperref[CL]{Connectedness Lemma}, and that one, and hence both, of these fundamental groups must be trivial or $\Z_2$ by Synge's theorem. Combined with Theorem \ref{thm:cohomology-to-homotopy}, it suffices to show that the universal cover of $N^m$, $\widetilde N^m$, is a cohomology sphere. We lift  to the universal covers $\widetilde M$ and $\widetilde N$ as in the proofs of the previous two lemmas. 

If Part (1) holds, then $\widetilde N \subseteq \widetilde M$ is also $m$-connected. 
By the \hyperref[PL]{Periodicity Lemma}, we have both that $e$ generates $H^3(\widetilde M; \Z)$ and that the map $H^3(\widetilde M;\Z) \to H^6(\widetilde M;\Z)$ given by multiplication by $e$ is injective. Combined with the graded commutativity of the cup product, we find that $2e = 0$ and hence that $H^3(\widetilde M; \Z)$ is either zero or $\Z_2$. 
In the latter case, the image $\bar e$ under the natural map $H^3(\widetilde M; \Z) \to H^3(\widetilde M; \Z_2)$ is non-trivial by the Bockstein sequence and induces three-periodicity in $H^*(\widetilde M; \Z_2)$. 
By the $\Z_2$-periodicity theorem \cite[Proposition 1.3]{Ken13}, there is moreover an element of degree one inducing periodicity in $H^*(\widetilde M; \Z_2)$. 
But $H^1(\widetilde M; \Z_2) = 0$ because $\widetilde M$ is simply connected, so $\widetilde M$ must be a $\Z_2$-homology sphere in contradiction to the fact that $\bar e \neq 0$, so this case does not occur. 
In the former case, where $H^3(\widetilde M; \Z) = 0$, we have that $e = 0$. Since $e$ induces periodicity, $\widetilde M$ is a homology sphere. By the \hyperref[CL]{Connectedness Lemma}, the same holds for $\widetilde N$.

To prove the result assuming Part (2), let $F = N \cap N'$. 
Again we lift to the universal covers $\widetilde M$, $\widetilde N$, $\widetilde N'$, and $\widetilde F$.

Assume first that the $\Z_2^2$-action on the normal space to $F$ has exactly two pairwise inequivalent irreducible subrepresentations. The submanifolds $N$ and $N'$ intersect transversely, and the same is true for $\widetilde N$ and $\widetilde{N'}$. Assuming $\codim(\widetilde{N'}) \geq 3$, the inclusion $\widetilde F \subseteq \widetilde{N'}$ is $\dim(\widetilde F)$-connected by the \hyperref[CL]{Connectedness Lemma}, $\widetilde{N'}$ is a homology sphere by Part (1), $\widetilde M$ is a homology sphere by Proposition \ref{pro:purple}, and finally $\widetilde N$ is a homology sphere by the \hyperref[CL]{Connectedness Lemma}. If $\codim(\widetilde{N'}) < 3$, then one argues similarly starting with the inclusion $\widetilde F \subseteq \widetilde N$ and applying either the \hyperref[Codim1]{Codimension One Lemma} or \hyperref[codim2]{Codimension Two Lemma} to that inclusion.

Assume instead that the $\Z_2^2$ action has three pairwise inequivalent irreducible subrepresentations. In this case, there is a third fixed-point set component $N''$ such that $\codim(F \subseteq N')$ and $\codim(F\subseteq N'')$ are positive integers summing to three by Borel's formula. In particular, $\widetilde F$ is a homology sphere by the \hyperref[Codim1]{Codimension One Lemma} and it follows that $F\subset N'$ is at least $(\dim(\widetilde F)-1)$-connected by the \hyperref[CL]{Connectedness Lemma}. Using a similar argument as in the first case where $\codim(N')\geq 3$,  we see that $\widetilde{N'}$, $\widetilde{M}$, and $\widetilde N$ are homology spheres.

\end{proof}

Before we proceed, we establish some notation. We denote by $E_\ell^{4m+2}$ a closed, simply manifold whose four-periodic cohomology ring is determined by the conditions that some $y \in H^2(M;\Z)$ and $x \in H^4(M;\Z)$ are generators and $y^2 = \ell x$ for some non-negative integer $\ell$. Note that $E_0$ is a cohomology $S^2 \times \HP^m$ and  $E_1$ is a cohomology $\CP^{2m+1}$.

Similarly, we denote by $N_j^{m}$ a closed, simply connected manifold whose four-periodic cohomology ring is determined by the  following conditions when $m \equiv 2 \bmod 4$:
\begin{equation}\label{e:1} 
\begin{matrix}
H^1(N_j;\Z) = H^2(N_j;\Z) = 0 &\mathrm{~~and~~} \\ H^3(N_j;\Z) = H^4(N_j;\Z) = \Z_j &\textrm{ for some} \,\, j \geq 2,
\end{matrix}
\end{equation}
	 and by the following conditions when $m \equiv 3 \bmod 4$:
	 \begin{equation}\label{e:2}
	 \begin{matrix}
	H^1(N_j;\Z) = H^2(N_j;\Z) = H^3(N_j;\Z) = 0 \mathrm{~~and~~} \\
	H^4(N_j;\Z) = \Z_j  \textrm{ for some} \,\, j \geq 2.
	\end{matrix}
	\end{equation}

We note that while two of the conclusions of the following lemma include $E_{\ell}$ with $\ell \geq 2$ or $N_j$ with $j \geq 2$, it is not known whether such manifolds actually exist nor if they do exist whether they admit a metric of positive sectional curvature.

\begin{Codim4}\phantomsection\label{codim4} 
Suppose $N^m \subseteq M^{m+4}$ is an $m$-connected inclusion of closed, positively curved manifolds. If $m \geq 4$, then one of the following holds:
	\begin{enumerate}
	\item $N^m$ is a homotopy $S^m/\Z_k$ for some $k \geq 1$, $\CP^{\frac m 2}$, or a cohomology $\HP^{\frac m 4}$;
	\item $\widetilde N^m$ has the cohomology of $E_\ell^m$ or $N_j^{m}$, and $m \equiv 2 \bmod 4$;  or 
	\item $\widetilde N^m$ has the cohomology of $S^3 \times \HP^{\frac{m-3}{4}}$ or $N_j^{m}$, and $m \equiv 3 \bmod 4$.  
	\end{enumerate}
Moreover, the same conclusions hold for $M$, with a slight modification to the formulas for the dimensions. 
\end{Codim4}

\begin{remark} The simply connected case is proved for $m \equiv 0,1 \bmod 4$ in Lemma 7.1 in Wilking \cite{Wil03}. For the cases $m \equiv 2,3\bmod 4$, we note that  Proposition 7.4 in  \cite{Wil03} neither implies nor is implied by these cases, but the statements are consistent. \end{remark}

\begin{proof} 
The computation of the cohomology groups of the universal cover of $N^m$ follows by combining Poincar\'e Duality with the definition of four-periodicity. To finish the proof, we need to consider the fundamental group.

Note that for the cases where $\widetilde N$ is a cohomology $\CP^{\frac m 2}$ with $m \equiv 0 \bmod 4$ or $\HP^{\frac m 4}$, it is known for topological reasons that there is no free action by any non-trivial group (see Example 4L.4 in Hatcher \cite{Hat02}). Hence $\pi_1(N)$ is trivial in these cases, and Conclusion (1) holds. For the case of $\CP^{\frac m 2}$ with $m \equiv 2 \bmod 4$, this falls under Conclusion (2).

Suppose now that $\widetilde N$ is a cohomology sphere. If $m$ is even, then $\pi_1(N)$ is either trivial or isomorphic to $\Z_2$ by Synge's theorem. If instead $m$ is odd, then we may argue as in Corollary 1.9 in \cite{FRW13} to conclude that $\pi_1(N)$ is cyclic. It then follows that $N$ is homotopy equivalent to $S^m/\Z_k$ for some $k \geq 1$ by Theorem \ref{thm:cohomology-to-homotopy}.
\end{proof}

\smallskip\section{Proof of Theorem \ref{thm:n}}\label{sec:n}

Our goal in this section is to prove Theorem \ref{thm:n}. The starting point is the following result due to Fang and Grove (see \cite{FG}): 
\begin{theorem}\cite{FG}\label{thm:FangGrove}
If a closed, positively curved manifold $M^n$ admits an isometric $\Z_2$-action with fixed-point set of codimension one, then $M$ is diffeomorphic to $S^n$ or $\RP^n$.

Moreover, if the $\Z_2$-action extends to an isometric action by a group $W$ generated by involutions each with this property, then there is a linear $W$-action on $S^n$ or $\RP^n$ such that the diffeomorphism may be chosen to be equivariant to a linear $\Z_2$-action.
\end{theorem}

To prove Theorem \ref{thm:n}, it suffices to prove that a $\Z_2^r$-action on $M^n$ having a fixed point has the property that the $\Z_2^r$ is generated by elements whose fixed-point set is of codimension one. This is straightforward, but we include the proof for the sake of completeness.

\begin{lemma}\label{obs} Let $M^n$ be an $n$-dimensional closed, positively curved Riemannian manifold admitting an effective, isometric $\mathbb{Z}_2^r$-action 
fixing $p\in M$. Then $r\leq n$, and in the case of equality, there exist involutions $\iota_1,\ldots,\iota_n$ generating $\Z_2^n$ such that $\codim(M^\iota_x) = 1$ for all $i$.
\end{lemma} 

\begin{proof} 
Note that the isotropy representation at $x$ 
is an injective map $\rho:\Z_2^r\embedded\gO(T_x M)\cong\gO(n)$. Moreover, the images 
of elements of $\Z_2^r$ under $\rho$ are 
symmetric matrices, as they are orthogonal matrices of order two. Since these symmetric matrices commute, they are simultaneously diagonalizable. This means that we may conjugate $\rho$
so that its image lies in the subgroup $\Z_2^n\subseteq\gO(n)$ of diagonal matrices 
of order two. Since $\rho$  
is injective, we have $r\leq n$. 
Now, in the case of equality, after possibly pre-composing with an automorphism of $\Z_2^n$, we see that  $\rho$  
is the identity map. In particular, there exists a generating set $\iota_1,\ldots,\iota_n \in \Z_2^n$ such that each $\iota_i$ fixes exactly $n-1$ directions in the tangent space $T_x M$ and hence has fixed-point set of codimension one.\end{proof}

\smallskip\section{Proof of Theorem \ref{thm:nover2}}\label{sec:nover2}

We begin this section with the  proofs of two results on the fundamental group, the first of which is used in the proof of Theorem \ref{thm:nover2}.  
The proofs illustrate again the utility of induction and error correcting codes in positive curvature. We finish this section with a proof of Theorem \ref{thm:nover2}.

\subsection{Fundamental Group Results}
We first prove a  Maximal $\Z_2$-Symmetry Rank lemma for lens spaces.
\begin{lemma} \label{lem:Z2MSRlens}
If $\Z_2^r$ acts isometrically on a closed, positively curved manifold $M^n$ with $|\, \pi_1(M)| > 2$, and if the action has a fixed point, then $r \leq \tfrac{n+1}{2}$.
\end{lemma}

We note that equality is achieved by lens spaces in all odd dimensions (see Section \ref{sec:Examples}).

\begin{proof}[Proof of Lemma \ref{lem:Z2MSRlens}]
Note by Synge's theorem that the result holds vacuously for even dimensions. 
In odd dimensions, it suffices to show that $r \geq \tfrac{n+3}{2}$ implies $|\pi_1(M)| \leq 2$. This holds for $n = 3$ by Theorem \ref{thm:n}, so we proceed by induction over $n$ and assume $n \geq 5$. 

By Part (1) of Lemma \ref{lem:ECC}, the bound on $r$ implies the existence of a fixed-point set component $N$ of an element $\iota\in \Z_2^r$ with the property that $\codim(N) \leq \tfrac{n+3}{4}$. In particular, since $n\geq 5$,  
we have $\pi_1(M) \cong \pi_1(N)$ by the \hyperref[CL]{Connectedness Lemma}. As the result follows by the \hyperref[Codim1]{Codimension One Lemma} if $\codim(N) = 1$, we assume $\codim(N) \geq 2$.

We also assume that $N$ is maximal, and so $N$ admits an effective, isometric $\Z_2^{r-1}$ action with a fixed point. Combining the lower bounds  on $r$ and $\codim(N)$, we see that the induction hypothesis applies to $N$. Hence $|\pi_1(M)| = |\pi_1(N)| \leq 2$.
\end{proof}

We now prove a Maximal $\Z_2$-Symmetry Rank lemma for space forms. 
\begin{lemma}
\label{lem:Z2MSRform}
If $\Z_2^r$ acts effectively, isometrically, and with a fixed point on a closed, positively curved manifold $M^n$ such that $\pi_1(M)$ is not cyclic, then $r \leq \tfrac{n+5}{4}$.
\end{lemma}

We remark that equality is achieved by spherical space forms $S^{4r-5}/\Gamma$ for any finite subgroup $\Gamma \subseteq \Sp(1)$ (see Section \ref{sec:Examples}). Since $\Gamma$ may be chosen to be non-cyclic, the bound on $r$ is optimal.

\begin{proof}
As in the  proof of Lemma \ref{lem:Z2MSRlens}, it suffices to show that if $n$ is odd and $r \geq \tfrac{n+7}{4}$, then $\pi_1(M)$ is cyclic. This holds for $n = 3$ by Theorem \ref{thm:n}, so we assume that $n \geq 5$ and proceed by induction on $n$.

By Part (2) of Lemma \ref{lem:ECC}, there is a fixed-point set component of some $\iota\in \Z_2^r$ with codimension $\leq \tfrac{n-1}{2}$. Without loss of generality, we may assume that $N$ has maximal dimension and hence admits an induced isometric and effective $\Z_2^{r-1}$ action with a fixed point. If $\codim(N) \in \{1,3\}$, then $\pi_1(M)$ is trivial or $\Z_2$ by Synge's theorem and the \hyperref[CL]{Connectedness Lemma}, and if $\codim(N) = 2$, then $\pi_1(M)$ is cyclic by Part (1) of the \hyperref[codim2]{Codimension Two Lemma}. We may therefore assume $\codim(N) \geq 4$. Using the lower bounds on $r$ and $\codim(N)$, we see that the induction hypothesis applies to $N$, so $\pi_1(N)$ is cyclic. The same now holds for $\pi_1(M)$ by the \hyperref[CL]{Connectedness Lemma} since $\codim(N)\leq \tfrac{n-1}{2}$.
\end{proof}

\subsection{The Proof of Theorem \ref{thm:nover2}}

We are now in a position to prove Theorem \ref{thm:nover2}, which is an immediate consequence of Theorem \ref{thm:nover2PLUS} below.

Recall that we set \[B = \{4, 12\} \cup \{3, 7, 11, 23\},\]
and defined $\delta_B(n)$ to be one for $n \in B$ and zero otherwise. Theorem \ref{thm:nover2} is a consequence of the following theorem.

\begin{theorem}\label{thm:nover2PLUS}
Let $M^n$ be a closed, positively curved manifold with $n \geq 2$.
Assume $\Z_2^r$ acts on $M^n$ and has a fixed point.
If $r > \tfrac{n}{2} + \delta_B(n)$, then one of the following holds:
	\begin{enumerate}
	\item $M$ is homotopy equivalent to $S^n$ or $\RP^n$. 
	\item $M$ is homotopy equivalent to $\CP^{\frac n 2}$, $r = \tfrac{n}{2} + 1$, and $n\not\in B$.
	\item $M$ is homotopy equivalent to $S^n/\Z_{\ell}$ for some $\ell \geq 3$, $r = \tfrac{n+1}{2}$, and $n\not\in B$.
	\end{enumerate}
\end{theorem}

\begin{proof}
We begin by outlining the strategy for the proof. First, we compute the cohomology of the universal cover $\widetilde M$ and the fundamental group of $M$ and show that it satisfies the hypotheses of Theorem \ref{thm:cohomology-to-homotopy}. The conclusions on the homotopy type of $M$ then follow. 

Second, given the homotopy type, we obtain the values for $r$ in Parts (2) and (3) as follows: in Part (2), $r>\tfrac n 2$ combined with the upper bound of Part (2) of Theorem \ref{thm:curvature_free} imply $r = \tfrac{n}{2} + 1$; similarly, for Part (3), $r>\tfrac n 2$ combined with the upper bound in Lemma \ref{lem:Z2MSRlens} gives us $r=\tfrac{n+1}{2}$.

Third, the cases $n \in \{2,3,4\}$ hold by Theorem \ref{thm:n}, so we assume that $n \geq 5$ and proceed by induction over $n$.

Let $x$ be a fixed point. By Part (1) of Lemma \ref{lem:ECC}, there exists a non-trivial involution $\iota \in \Z_2^r$ and a component $N^{n-k}$ of the fixed-point set of $\iota$ such that $k \leq \tfrac{n+3}{4}$. In particular, this implies that $N$ admits an induced isometric, effective action by a $\Z_2^{r-1}$ with a fixed point.

We now proceed to argue according to the codimension of $N$. There are four cases.

\smallskip\noindent{\bf Case 1:} Suppose $k = 1$.

Part (1) holds by the \hyperref[Codim1]{Codimension One Lemma}.

\smallskip\noindent{\bf Case 2:} Suppose $k = 2$ and $n$ is odd. 
Part (1) of the \hyperref[codim2]{Codimension Two Lemma} implies that $\widetilde M$ is a sphere and that $\pi_1(M)$ is cyclic. Hence Part (1) or Part (3) holds by the comments at the beginning of the proof.

\smallskip\noindent{\bf Case 3:} Suppose $(k, n) \in \{(2,6), (2,14), (3,15)\}$. 

Notice in these cases that $n - k \in \{4, 12\}$ and that we cannot apply the induction hypothesis to $N$. 
From the bound on $r$, we have $r \geq 4$ in the first case and $r \geq 8$ in the second and third cases. 
By Part (1) of Lemma \ref{lem:ECCshorten}, there exists a second involution $\iota_2$ independent from $\iota$
such that its fixed-point set component containing $x$ has codimension at most $2$ in the first case and at most $6$ in the other two cases. 

Applying either Part (3) of the  \hyperref[codim2]{Codimension Two Lemma} or Part (2) of the \hyperref[codim3]{Codimension Three Lemma}, we find that $\widetilde N$ is a cohomology sphere or complex projective space.
Applying Theorem \ref{thm:purple-OneSubmanifold}, 
we find that $\widetilde M$ is a cohomology sphere or complex projective space as well. Finally, $\pi_1(M)$ is trivial or $\Z_2$ by Synge's theorem, together with the fact that $\pi_1(M^{15}) \cong \pi_1(N^{12})$ in the case where $(k, n) = (3,15)$.

\smallskip\noindent{\bf Case 4:} We now consider all the other possibilities. 

We claim that $r - 1 > \tfrac{1}{2} \dim(N) + \delta_B(\dim(N))$, in which case, the induction hypothesis applies to $N$, and we have that Part (1), (2), or (3) holds for $N$. In particular, $\widetilde N$ has the cohomology of a sphere or complex projective space, and the same holds for $\widetilde M$ by the \hyperref[CL]{Connectedness Lemma}. In addition, $\pi_1(M) \cong \pi_1(N)$ is trivial in the latter case and cyclic in the former. By the \hyperref[CL]{Connectedness Lemma} and the estimate $\codim(N) \leq \tfrac{n+3}{4}$, the topology of $\widetilde M$ follows. 

To finish the proof, it suffices to prove the claim on the lower bound for $r-1$. 
The estimate is straightforward if $k \geq 4$. 
If $k = 3$ and $n$ is even, then the fact that $r$ is an integer implies $r \geq \tfrac{n}{2} + 1 + \delta_B(n)$, so the proof follows similarly.
If $k = 3$ and $n$ is odd, the proof is similar unless $\delta_B(n-3) = 1$. But this latter case does not occur since otherwise $(k, n)= (3,15)$ and we are back in Case (3). Next if $k \leq 2$, we may assume $k = 2$ and $n$ is even by Cases (1) and (2). Moreover we may assume $n - 2 \not\in \{4,12\}$ by Case (3), so $\delta_B(n-2) = 0$ and the claim again follows.
\end{proof}

\smallskip\section{Proof of Theorem \ref{thm:nover4}}\label{sec:nover4}

The main result of this section is Theorem \ref{thm:nover4PLUS}. It implies Theorem \ref{thm:nover4} for dimensions $n > 17$, as we show in the remarks at the end of this section. For dimensions $15$, $16$, and $17$, Theorem \ref{thm:nover4} follows by Theorem \ref{thm:nover4-smalldim}, which we prove later in this section.

In the following, $M \sim_* N$ means that $\pi_1(M) \cong \pi_1(N)$ and $H^*(\widetilde M;\Z) \cong H^*(\widetilde N;\Z)$. Also recall that the $E_\ell$ and $N_j$ listed in Conclusions 2 and 3 of Theorem \ref{thm:nover4PLUS} are, if they exist, manifolds that are closed, simply connected and have cohomology ring determined as in Displays  \eqref{e:1} and \eqref{e:2}. 
\begin{theorem}\label{thm:nover4PLUS}
Let $M^n$ be a closed, positively curved manifold, and assume $\Z_2^r$ acts effectively by isometries on $M$ with fixed point $x$. Fix some $j \in \{0,1,2,3,4\}$, and assume $F_j^{m_j}$ is the fixed-point set component at $x$ of some subgroup of $\Z_2^r$ such that the induced $\Z_2^r$-action on $F_j^{m_j}$ has kernel of rank $r - j$. If
	\[n > 3j+5 \hspace{.5in} \textrm{and} \hspace{.5in}  r \geq \frac{n+3+j}{4},\]
then one of the following holds:
	\begin{enumerate}
	\item $F_j^{m_j}$ is homotopy equivalent to $S^{m_j}$, $\RP^{m_j}$, $\CP^{\frac{m_j}{2}}$, or $S^{m_j}/\Z_q$ with $q \geq 3$;
	\smallskip\item $j = 2$ and $M \sim_* S^{4r-5}/\Gamma$, $(S^3/\Gamma) \times \HP^{r - 2}$, or $N_k^{4r-5}/\Gamma$ for some three-dimensional space form group $\Gamma$;
	\smallskip\item $j = 3$ and $M \sim_* E_\ell^{4r-6}/\Delta$ or $N_k^{4r-6}/\Delta$, where $\Delta$ is  trivial or isomorphic to $\Z_2$; or
	\smallskip\item $j = 4$ and $M \sim_* \HP^{r-2}$.
	\end{enumerate}
\end{theorem}

We remark that unlike the proof of Theorem \ref{thm:nover2}, the proof of Theorem \ref{thm:nover4PLUS} does not use induction over the dimension. 
We also remark that the examples $\HP^{r-2}$, $E_1^{4r-6}/\Z_2 = \CP^{2r-3}/\Z_2$, and $S^{4r-5}/\Gamma$ for $\Gamma \subseteq \Sp(1)$ in the $j =4$, $j =3$, and $j =2$ cases, respectively, can be  realized by positively curved manifolds with $\Z_2^r$ symmetry satisfying the hypotheses of Theorem \ref{thm:nover4PLUS} (see Section \ref{sec:Examples}). However, we do not know if there are examples of the last type for other space form groups $\Gamma$ or for the other topological types shown in the $j = 2$ and $j = 3$ cases.

\begin{proof}[Proof of Theorem \ref{thm:nover4PLUS}]
Fix $j \leq 4$ and $F_j^{m_j}$ as in the theorem. We are given that $F_j$ is the fixed-point set component at $x$ of a subgroup $\Z_2^{r-j} \subseteq \Z_2^r$ and admits an isometric and effective $\Z_2^j$ action that fixes $x$. 

Note that Part (1) holds if $m_j< 3$ by the classification of manifolds in these dimensions and the Gauss-Bonnet theorem. Part (1) also holds for  $m_3=3$ and $m_4\leq 6$ by Theorem \ref{thm:n} or Theorem \ref{thm:nover2}, so we may assume  
that 
\begin{equation}\label{mj}\begin{cases}
 m_j \geq 3 & 0\leq j\leq 2\\
 m_j\geq 4 & j=3\\
 m_j\geq 7 &j=4.\\
 \end{cases}
 \end{equation}

Choose $\Z_2^{r-(j+1)}\subseteq\Z_2^{r-j}$ such that the fixed-point set component $F_{j+1}^{m_{j+1}} \subseteq M^{\Z_2^{r-(j+1)}}$ containing $x$ has the property that
	\[k_j = \codim(F_j \subseteq F_{j+1})\]
is positive and minimal. Note that Conclusion (1) holds if $k_j = 1$ by the \hyperref[Codim1]{Codimension One Lemma}, so we assume $k_j \geq 2$.
Similarly, if $k_j = 2$, then Conclusion (1) follows by Part (1) of the \hyperref[codim2]{Codimension Two Lemma} or by a combination of Theorem \ref{thm:nover2} and the Connectedness lemma if $m_{j+1} = 5$ and $j \leq 2$, if $m_{j+1} = 6$ and $j = 3$, or if $m_{j+1} = 9$ and $j = 4$. 
Thus,  
we may assume 
\begin{equation}\label{mj+1}\begin{cases}
 m_{j+1} \geq 6 & 0\leq j\leq 2\\
 m_{j+1}\geq 7 & j=3\\
 m_{j+1}\geq 10 &j=4.\\
 \end{cases}
 \end{equation}

To complete the proof, we consider two cases 
according to the size of $k_j$. 
We remark that, in Case 1, we show that Part (1) holds.

\bigskip\noindent{\bf Case 1:} $2\leq k_j \leq 3$. 

Choose $\Z_2^{r-(j+2)} \subseteq \Z_2^{r-(j+1)}$ such that the fixed-point set component 
$F_{j+2}^{m_{j+2}} \subseteq M^{\Z_2^{r-(j+2)}}$ containing $x$ has the property that
	\[k_{j+1} = \codim(F_{j+1} \subseteq F_{j+2})\]
is positive and minimal. 

Choose $\Z_2^2 \subseteq \Z_2^r$ acting effectively on $F_{j+2}$ and
fixing $F_j$. If the number of pairwise inequivalent irreducible subrepresentations is two, then Observation \ref{borel} implies that the inclusion $F_j \subseteq F_{j+1}$ is $m_j$-connected. Applying either Part (2) of the \hyperref[codim2]{Codimension Two Lemma} or Part (1) of the \hyperref[codim3]{Codimension Three Lemma} shows that Conclusion (1) holds. We may assume therefore that this $\Z_2^2$ representation has exactly three pairwise inequivalent subrepresentations. Applying the \hyperref[Borel]{Borel Formula}, together with the minimality of $k_j$, we see that $k_{j+1} \geq 2k_j$.

Now, consider the isotropy representation of $\Z_2^{r-(j+1)}$ on the normal space to $F_{j+1}$, and let $q$ denote the number of pairwise inequivalent irreducible subrepresentations.  
We claim that $q = r - j - 1$ and
$k_{j+1} \leq 5$. 
Indeed, if $q \geq r - j$, then the \hyperref[Borel]{Borel Formula} implies that
	\[n - m_{j+1} \geq k_{j+1} q \geq 2 k_j (r - j) \geq 4\of{\frac{n+3-3j}{4}},\]
which contradicts the lower bound on $m_{j+1}$ in Display \ref{mj+1}. Similarly, if
$k_{j+1} \geq 6$, then the \hyperref[Borel]{Borel Formula} and the assumption $n > 3j+5$ imply
	\[n - m_{j+1} \geq k_{j+1} q \geq 4(r-j-1) + 2(r-j-1) \geq (n-1-3j) + 2(2),\]
which again contradicts the lower bound on $m_{j+1}$. 
Equipped with the facts that $q = r - j - 1$ and $k_{j+1} \leq 5$, we are ready to complete the proof in Case 1. 
Note in particular that the upper bound on $k_{j+1}$ implies that $k_j = 2$ and $m_{j+1} \geq k_{j+1} + 1$.

Look at the inclusion $F_{j+1} \subseteq F_{j+2}$. By Observation \ref{borel}, this inclusion is $m_{j+1}$-connected. Since $m_{j+1} \geq k_{j+1} + 1$, the \hyperref[PL]{Periodicity Lemma} implies that the universal cover $\widetilde{F_{j+2}}$ is $k_{j+1}$-periodic.

Suppose first that $m_{j+1} = k_{j+1} + 1$. In particular, we have $m_{j+1} = 6$ and $k_{j+1} = 5$ by the lower bound on $m_{j+1}$ and the upper bound on $k_{j+1}$. In addition, we see that $k_{j+1}$ divides $\dim (\widetilde{F_{j+2}}) - 1$, so Lemma 7.1 of \cite{Wil03} implies that $\widetilde{F_{j+2}}$ is a cohomology $S^{11}$. The \hyperref[CL]{Connectedness Lemma} then implies that $\widetilde{F_{j+1}}$ is a cohomology $S^6$ and that $\widetilde{F_j}$ is a cohomology $S^4$. Synge's theorem and Theorem \ref{thm:cohomology-to-homotopy} 
then imply that $F_4$ is homotopy equivalent to $S^4$ or $\RP^4$.

Suppose then that $m_{j+1} \geq k_{j+1} + 2$. The Connectedness and Periodicity Lemmas applied to the inclusion $\widetilde{F_j} \subseteq \widetilde{F_{j+1}}$ imply that the multiplication map
	\[H^{k_{j+1}-k_j}(\widetilde{F_{j+1}}) \to H^{k_{j+1}}(\widetilde{F_{j+1}})\]
is surjective. Together with the fact that the latter group is isomorphic to $H^{k_{j+1}}(\widetilde{F_{j+2}})$ by the \hyperref[CL]{Connectedness Lemma}, we find that the element in $H^*(\widetilde{F_{j+2}})$ inducing $k_{j+1}$-periodicity factors as the product of elements of degree two and $k_{j+1}-2$. Hence $\widetilde{F_{j+2}}$ has $2$-periodic cohomology, which implies it is a cohomology sphere or complex projective space.

Since its inclusion into $\widetilde{F_{j+2}}$ is $m_{j+1}$-connected, $\widetilde{F_{j+1}}$ is also a cohomology sphere or complex projective space. And since it has dimension $m_j \geq 3$ and codimension $k_j = 2$ in $\widetilde{F_{j+1}}$, $\widetilde{F_j}$ is similarly a cohomology sphere or complex projective space. In particular, the inclusion $\widetilde{F_j} \subseteq \widetilde{F_{j+1}}$ is $m_j$-connected and likewise for the inclusion $F_j \subseteq F_{j+1}$, so Part (1) of the \hyperref[codim2]{Codimension Two Lemma} implies that $F_j$ is homotopy equivalent to a sphere, a real or complex projective space, or a lens space.

\bigskip\noindent{\bf Case 2:} $k_j \geq 4$.

We apply the \hyperref[Borel]{Borel Formula} to the $\Z_2^{r-j}$ representation on the normal space at $x \in F_j$. By the minimality of $k_j$, we have 
\begin{equation}\label{basic}
\codim(F_j) \geq k_j(r-j) \geq 4(r-j)\geq n + 3 - 3j.
\end{equation}
In addition, if we assume that $k_j \geq 5$ or if the number of pairwise inequivalent irreducible subrepresentations is at least $r - j + 1$, then the estimate holds with an additional $r-j \geq 3$ or $k_j \geq 4$, respectively, added to the right-hand side of Display \ref{basic}. In either of these cases, we find that 
 $j\geq 3$, as otherwise $F_j$ either has negative dimension or dimension zero and an effective $\Z_2$ action. We similarly find that $m_j \leq 3$ if $j = 3$ and that $m_j \leq 6$ if $j = 4$, so Conclusion (1) holds, as noted earlier.

Therefore, we may assume that the isotropy representation at $x \in F_j$ has exactly $r - j$ such subrepresentations, and $k_j=4$. 
Going back to the estimate in Display \ref{basic}, 
we find that $j \in\{2,3,4\}$ since otherwise $F_j$ either has negative dimension or dimension zero and an effective $\Z_2$ action. We complete the proof by considering the possibilities for $j$.

First, suppose that $j = 2$. This implies that $m_2 \leq 3$. Since Conclusion (1) holds if $m_2 \leq 2$, we may assume that $m_2 = 3$ and hence that equality holds everywhere in the estimate in Display \ref{basic}. In particular, the isotropy representation on the normal space to $F_2^3$ takes the form
$$(\ep_1,\ldots,\ep_{r-2}) \mapsto (\ep_1 I_4,\ldots, \ep_{r-2} I_4),$$ 
where $\ep_i= \pm 1$ for $1\leq i\leq r-2$.
The \hyperref[CL]{Connectedness Lemma} and the \hyperref[PL]{Periodicity Lemma} imply that $\widetilde F_{4}$ has four-periodic cohomology, and the \hyperref[codim4]{Codimension Four Lemma} implies that $\widetilde F_{4}$ has the cohomology of $S^{m_4}$, $S^3 \times \HP^{\frac{m_{4} - 2}{4}}$ or $N_j^{m_{4}}$ as in the statement of the \hyperref[codim4]{Codimension Four Lemma}. By the \hyperref[CL]{Connectedness Lemma}, the same conclusions hold for $F_i$ for all $i >4$ and in particular for $M$. Since in addition, $\pi_1(M) \cong \pi_1(F_2^3)$ by the  \hyperref[CL]{Connectedness Lemma}, we find that $\pi_1(M)$ is a three-dimensional spherical space form group by the classification of positively curved three-dimensional manifolds (see \cite{Ham82}), so Conclusion (2) holds.

Second, we consider the case $j = 3$. Display \ref{basic} implies that $m_3 \leq 6$. Part (1) holds if $m_3 \leq 3$, as we saw earlier, or if $m_3 = 5$, by Theorem \ref{thm:nover2PLUS}. So we assume now that $m_3 = 4$ or $m_3 = 6$. 
In the former case, $F_3^4 \subseteq F_{4}^8$ is $4$-connected by Observation \ref{borel} and hence $F_4$ is a cohomology $S^8$, $\RP^8$, $\CP^4$, or $\HP^2$ by the Codimension Four Lemma. In the first three cases, Conclusion (1) holds immediately. In the last case, it follows that $F_3^4$ is $3$-connected and hence is a cohomology $S^4$, so again Conclusion (1) holds.
Suppose then that $m_3 = 6$. We have equality everywhere in Display \ref{basic}, and the proof follows with only slight modifications as in the previous paragraph, and we find that Conclusion (3) holds.

Third, assume $j = 4$ and $m_4 \neq 7$. The estimate in Display \ref{basic}, together with our assumptions that $m_4\geq 7$ and $m_4 \neq 7$, implies that $m_4 \in \{8,9\}$. If $m_4 = 9$, then the \hyperref[codim4]{Codimension Four Lemma} implies Conclusion (1) since $m_{5} = 13$ is congruent to one modulo four. Suppose then that $m_4 = 8$. The \hyperref[codim4]{Codimension Four Lemma} implies that either $F_4$ is a homotopy $S^8$, $\RP^8$, or $\CP^4$ or that $F_4$ is a cohomology $\HP^2$. In the first three cases, we have Conclusion (1), so we may assume we are in the last case. The inclusion $F_4 \subseteq F_{5}$ is $8$-connected by Observation \ref{borel}, so $F_{5}$ is a cohomology $\HP^3$. Using again that $r-j \geq 3$, we obtain further that $F_{6}$ is a cohomology $\HP^4$ by Part 2 of the \hyperref[CL]{Connectedness Lemma}. Finally, applying Theorem \ref{thm:purple-OneSubmanifold} to the inclusions $F_i \subseteq F_{i-1}$ implies that all of these codimensions are four and that $\widetilde F_i$ is a cohomology $\HP^{i-2}$ for all $4 \leq i \leq r$. In particular, Conclusion (4) holds.

Finally, we assume that $j = 4$ and $m_4 =7$. Then $F_4^7\subseteq F_5^{11}$ and by Part (3) of Lemma \ref{lem:ECCshorten}, we see that for the induced $\Z_2^5$ action on $N^{11}=F_5^{11}$ either there exists a non-trivial $\iota_1 \in \Z_2^5$ such that $\codim(N^{\iota_1}_x) \leq 3$, or 
there exist distinct, non-trivial $\tau_1, \tau_2 \in \Z_2^5$ such that $N^{\tau_1}_x$ and $N^{\tau_2}_x$ have codimension four and intersect non-transversely. 
In the first case, we apply the \hyperref[Codim1]{Codimension One Lemma}, the \hyperref[codim2]{Codimension Two Lemma}, or the \hyperref[codim3]{Codimension Three Lemma} to conclude that  
$N^{\iota}_x$ has cyclic fundamental group and has a sphere as universal cover.
The same then holds for $F_5^{11}$ and $F_4^7$ 
by 
Theorem \ref{thm:purple-OneSubmanifold}, so Conclusion (1) follows by Theorem \ref{thm:cohomology-to-homotopy}. 
In the second case, we argue similarly to see that $\pi_1(F_4^7)$ is cyclic and that $\widetilde F_5^{11}$ is four-connected. We cannot prove that $\widetilde F_5^{11}$ is a sphere, but nevertheless we have enough to conclude that $\widetilde F_4^7$ is a sphere by the \hyperref[CL]{Connectedness Lemma}, so once again Conclusion (1) holds.
\end{proof}

\begin{theorem} \label{thm:nover4-smalldim}
Let $M^n$ be a closed, positively curved manifold, and assume $\Z_2^r$ acts effectively by isometries on $M$ with fixed point $x$. 
Assume $F_4^{m_4}$ is the fixed-point set component at $x$ of some subgroup of $\Z_2^r$ such that the induced $\Z_2^r$-action on $F_4^{m_4}$ has kernel of rank $r - 4$. If $n\in\{15, 16, 17\}$ and $r \geq \tfrac{n+7}{4}$
	then 
one of the following holds:
	\begin{enumerate}
	\item $F_4^{m_4}$ is homotopy equivalent to $S^{m_4}$, $\RP^{m_4}$, $\CP^{\frac{m_4}{2}}$, or $S^{m_4}/\Z_k$ with $k \geq 3$.
		\smallskip\item 
	$M \sim_* \HP^{r-2}$, with $r-2=4$.
	\end{enumerate}

\end{theorem}
\begin{proof} 
First observe that $r\geq 6$ for all values of $n$. Moreover, we may assume 
 that $m_4 \geq 7$,  as  the result holds by Theorem \ref{thm:nover2} for $m_4\leq 6$. 
 For simplicity, we pass to a subgroup $\Z_2^6 \subseteq \Z_2^r$. We modify the proof of Theorem \ref{thm:nover4PLUS}  slightly to cover the following two cases: Case 1, where the isotropy representation of  the kernel $\Z_2^2 \subseteq \Z_2^6$ of the induced action on $F_4$ has three non-trivial, pairwise inequivalent irreducible subrepresentations, and Case 2, where there are exactly two irreducible subrepresentations.

{\bf Case 1:} The isotropy representation of the kernel $\Z_2^2 \subseteq \Z_2^6$ of the induced action on $F_4$ has three non-trivial, pairwise inequivalent irreducible subrepresentations. By the \hyperref[Borel]{Borel Formula},
	\[10 \geq n - m_4 \geq 3k_4,\]
so we have $k_4 \leq 3$. Moreover, if $k_4 = 1$ or if $k_4 = 2$ and $m_4$ is odd, then Part (1) holds by the \hyperref[Codim1]{Codimension One Lemma} or Part (1) of the \hyperref[codim2]{Codimension Two Lemma}. This leaves the following cases:
	$$\begin{cases}
	k_4 = 2 &\mathrm{ and } \,\,\, m_4 \in \{8,10\};\,\,\, \mathrm{ and } \\
	k_4 = 3 &\mathrm{ and } \,\,\, m_4 \in \{7,8\}.
	\end{cases}$$
In all four of these cases, Part (2) of Lemma \ref{lem:ECCshorten}  implies the existence of another fixed-point set component $F_4'$ inside $F_5$ as in Part (3) of the \hyperref[codim2]{Codimension Two Lemma} when $k_4=2$ or Part 2 of the \hyperref[codim3]{Codimension Three Lemma} when $k_4=3$. Hence $F_4$ is a homotopy $S^{m_4}/\Z_l$ or $\CP^{\frac{m_4}{2}}$, as required for Conclusion (1).

{\bf Case 2:} The isotropy representation of $\Z_2^2$ has exactly two irreducible subrepresentations. By the \hyperref[Borel]{Borel Formula} 
$F_4 \subseteq F_5$ is an $m_4$-connected inclusion. There are three cases: $k_4 \leq 3$, $k_4=4$ and $k_4\geq 5$. 

If $k_4 \leq 3$, the topology of $F_4$ satisfies Part (1) by the \hyperref[Codim1]{Codimension One Lemma}, the \hyperref[codim2]{Codimension Two Lemma}, or the \hyperref[codim3]{Codimension Three Lemma}.

Next, suppose that $k_4 = 4$.  If $m_4 = 7$, then the  proof for  Case 2 when $m_4=7$ of Theorem \ref{thm:nover4PLUS} gives us that Part (1) holds. By the \hyperref[Borel]{Borel Formula}, $17\geq n=m_4+4+k_4'$. So, if $m_4 \geq 9$, then in fact $m_4 = 9$, $n=17$, and Part (1) follows by the \hyperref[codim4]{Codimension Four Lemma}. Finally if $m_4 = 8$ and Part (1) does not hold, then $F_4$ is a cohomology $\HP^2$ and $F_5$ is a cohomology $\HP^3$. Since $m_6 \leq n \leq 17$, it follows moreover that $F_6$ is a cohomology $\HP^4$ by Theorem \ref{thm:purple-OneSubmanifold},  $F_6 = M$, and  $r = 6$. Hence Part (2) holds.

To finish the proof,  suppose $k_4 \geq 5$. Since $17\geq n \geq m_4 + k_4 + k_4'$, we have $n = 17$, $k_4 = k_4' = 5$, and $m_4 = 7$.  Since the inclusion $F_4^7 \subseteq F_5^{12}$ is $m_4$-connected, we can now proceed exactly as in the proof of Part (1) of the \hyperref[codim3]{Codimension Three Lemma} to show that $\widetilde F_5$ and $\widetilde F_4$ are cohomology spheres. In particular,  $F_4$ is a cohomology sphere or a real projective space. Hence, once again Part (1) holds.
\end{proof}

As mentioned earlier, Theorems \ref{thm:nover4PLUS} and Theorem \ref{thm:nover4-smalldim}  almost imply Theorem \ref{thm:nover4} from the introduction, but some additional work is required.

\begin{proof}[Proof of Theorem \ref{thm:nover4}]
Let $M^n$ and $\Z_2^r$ be as in Theorem \ref{thm:nover4}. In particular, there is a fixed point $x$, and $r \geq \tfrac{n+3}{4} + 1$. We are also given a fixed-point set component $F^{m}$ at $x$ of some corank four subgroup $\Z_2^{r-4} \subseteq \Z_2^r$. Let $\Z_2^{r-j} \supseteq \Z_2^{r-4}$ denote the kernel of the induced $\Z_2^r$ action on $F^m$, and define $F_j^{m_j}$ to be $F^m$. 

If $n \leq 3j+5$, then the assumption in Theorem \ref{thm:nover4} that $n \geq 15$ implies that $j = 4$ and $n \in \{15,16,17\}$, so the result holds by Theorem \ref{thm:nover4-smalldim}.

If instead $n > 3j+5$, then all of the assumptions of Theorem \ref{thm:nover4PLUS} hold. Moreover, Conclusions (1) and (4) of Theorem \ref{thm:nover4PLUS} are exactly Conclusions (1) and (2) of Theorem \ref{thm:nover4}, so it suffices to show that Conclusions (2) and (3) do not occur. In fact, this is the case, since otherwise either $j = 2$ and $r = \tfrac{n+5}{4}$ or $j = 3$ and $r = \tfrac{n+6}{4}$, and this contradicts the assumption in Theorem \ref{thm:nover4} that $r \geq \tfrac{n+7}{4}$.
\end{proof}

\smallskip\section{Examples}\label{sec:Examples}
We conclude by exhibiting examples of $\Z_2$-torus actions on the spaces listed in Theorems \ref{thm:n}, \ref{thm:nover2}, and \ref{thm:nover4}. The table below summarizes the results of the constructions in this section. Observe that for each manifold in the first row, the value $r$ in the second row denotes the rank of a $\Z_2$-torus that acts effectively by isometries and has a fixed point. The groups $\Z_\ell$, $\Gamma$, and $\Gamma_p$ represent, respectively, a cyclic group of order at least three, a non-cyclic subgroup of $\Sp(1)$, and a non-cyclic $(2p-1)$-dimensional space form group of Type I (see Section \ref{sec:ExamplesC}).

\begin{table}[ht]\label{sum}\caption{Examples with  ${\bf \Z_2^r}$-Fixed Point Symmetry}
\[\begin{array}{|c||c||c|c||c|c|c||c|}\hline
\mathrm{Example}		& \thead{S^n, \RP^n} & \thead{\CP^{\frac n 2}}	&\thead{S^n/\Z_\ell} & \thead{\HP^{\frac n 4} }& \thead{\CP^{\frac n 2}/\Z_2} & \thead{S^n/\Gamma}	 & 
\thead{S^n/\Gamma_p}  \\\hline
\thead{r}	& \thead{n}		& \thead{\tfrac{n+2}{2}}	& \thead{\tfrac{n+1}{2}}	& \thead{\tfrac{n+8}{4}}	& \thead{\tfrac{n+6}{4}}	& \thead{\tfrac{n+5}{4}}		& \thead{\tfrac{n+1}{2p}} 	\\\hline
\end{array}
\]\end{table}

\subsection{The Maximal ${\bf \Z_2}$-Fixed Point Symmetry Rank Case}

We may also think of this case in terms of $S^{n}\subseteq \R^{n+1}$, on which $\gO(n+1)$ acts isometrically. Note that any isotropy subgroup contains a $\Z_2$-torus of rank $n$ fixing a line in $\R^{n+1}$ and hence a point in $S^n$ or $\RP^n$, respectively. 

\begin{example}[{\bf ${\bf S^n}$ and ${\bf \RP^n}$}]\label{sn}
Consider $S^n \subseteq \R^{n+1}$ equipped with the standard $\gO(n+1)$ action. Let $G\cong \Z_2^n$ be a subgroup of $\{\diag(\pm 1, \pm 1, \hdots, \pm 1)\} \cong \Z_2^{n+1} \subseteq \gO(n+1)$ with a $1$ replacing one of the $\pm1$ diagonal entries. Then, up to a change of basis, $G$ acts on $S^n$ fixing a point as follows:
$$\diag( 1, \pm 1, \hdots, \pm1) \cdot (x_0, x_1, \hdots, x_n)  =  ( x_0, \pm x_1, \hdots,\pm x_n).$$

This action commutes with the antipodal map on $S^n$ and induces a maximal effective $\Z_2^n$-action on $\RP^n$ with a fixed point.
\end{example}

\subsection{The Half-Maximal  ${\bf \Z_2}$-Fixed Point Symmetry Rank Case}\label{sec:ExamplesC}
Here, we  think of this case in terms of $S^{2m+1} \subseteq \C^{m+1}$ equipped with the standard $\gU(m+1) \rtimes \Z_2$ action, where the $\Z_2$-factor acts by complex conjugation. Note that $\gU(m+1) \rtimes \Z_2$ contains a maximal $\Z_2$-torus fixing a line in $\C^{m+1}$ and hence a point in $S^{2m+1}$, $\CP^m$, or a lens space. 

\begin{example}[{\bf $\CP^m$ and lens spaces}]\label{ex1}
Consider $S^{2m+1} \subseteq \C^{m+1}$ equipped with the standard $\gU(m+1) \rtimes \Z_2$ action, where the $\Z_2$-factor acts by complex conjugation. 
Let $H\cong \Z_2^{m}$ be a subgroup of $\{\diag(\pm 1, \pm 1, \hdots, \pm 1)\} \cong \Z_2^{m+1} \subseteq \gU(m+1)$ with a $1$ replacing one of the $\pm1$ diagonal entries. Then, up to a change of basis, $G \cong H \times \Z_2$ acts on $S^{2m+1}$ fixing a point as follows:
$$(\diag(1, \pm 1, \hdots, \pm 1), \ep) \cdot (z_1, z_2, \hdots, z_{m+1})  =   ( z_1^\ep, \pm z_2^\ep, \hdots, \pm z_{m+1}^\ep),$$
where $\ep\in \{\pm 1\}$ so that $z_i^\ep$ is either $z_i$ or $\bar{z}_i$.

This action commutes with the Hopf circle action on $S^{2m+1}$ and with any cyclic sub-action of odd order of the Hopf circle action. Thus it induces a maximal effective $\Z_2^{m+1}$-action on $\CP^m$ with a fixed point, as well as a maximal effective $\Z_2^{m+1}$-action on a $(2m+1)$-dimensional lens space with fundamental group of odd order with a fixed point. 
For lens spaces of even order, instead we consider the action given by $\Z_2^{m+1} \subseteq  \gU(m+1)$ modulo a diagonal $\Z_2$. This action commutes with any cyclic action of even order and induces a maximal effective $\Z_2^{m}$-action on a lens space with fundamental group of even order with a fixed point. Together with complex conjugation, we find the desired $\Z_2^{m+1}$-action once again.
\end{example}

Recall that the Type I spherical space forms, $S^{2d-1}/\Gamma$ have fundamental group for which every Sylow subgroup is cyclic and $\Gamma$ is defined as follows:
$$\Gamma\cong\langle\alpha, \beta|\alpha^a=\beta^b=1, \beta\alpha\beta^{-1}=\alpha^c\rangle,$$
for some $a, b, c\geq 1$ such that $\gcd(a, b) =\gcd(a, c-1) = 1$ and $c^b \equiv 1 \bmod a$ (see \cite{Wol}). To construct a free action of $\Gamma$ on $S^{2d-1}$ consider the representation 
 $\Gamma \subseteq \gU(d)$ given by:

\begin{align*}
\alpha&\mapsto \diag(e^{2 \pi i/a},  e^{2 \pi i c/a},\hdots e^{2 \pi i c^{d-1}/a})\\
\beta&\mapsto \begin{pmatrix} 0 & I\\
e^{2\pi i/d} &0\\
\end{pmatrix}.
\end{align*}
\medskip

Suppose that $m + 1 = rd$ for some integer $r$.  
Then we have an induced action of $\Gamma$  via the inclusion
	\[\Gamma \subseteq \gU(d) \subseteq \diag(\gU(d), \gU(d), \hdots, \gU(d)) \subseteq \gU(m+1)\]
that commutes with $G= \diag(\pm I_d, \pm I_d, \hdots, \pm I_d)\cong \Z_2^r$ given by the $d$-by-$d$ block embedding.  
Since we have treated the $d=1$ case in Example \ref{ex1}, we now consider the cases where $d\geq 2$. Note in this case that complex conjugation is not available, so the $\Z_2$-symmetry rank is one lower.

\begin{example}[{\bf Non-cyclic Type I spherical space forms}]\label{ex2} Let $H\cong \Z_2^{r-1}$ be a subgroup of $G =\{\diag(\pm I_d, \pm I_d, \hdots, \pm I_d)\} \subseteq \gU(m+1)$, with $I_d$ replacing one of the $\pm I_d$ diagonal entries. Then the induced $G$-action on $S^{2m+1}$ commutes 
with the $\Gamma$-action. If $\Gamma$ does not contain a $\Z_2$ subgroup, then $G$ induces an action on  $S^{2dr-1}/\Gamma$ with a fixed point.  As in the previous example, if there is a $\Z_2$ subgroup of $\Gamma$, then we consider a $\Z_2$ quotient of $G$ and again obtain a maximal, effective $\Z_2^{r}$-action on $S^{2dr-1}/\Gamma$ with a fixed point. 
\end{example}

\begin{remark}
If $d = 2$ in this example, then $\Gamma$ is a three-dimensional spherical space form group and we have a $\Z_2^r$ action with a fixed point on $S^{4r-1}/\Gamma$, with $r = \tfrac{\dim(M) + 1}{4}$. We see in the next set of examples that, for the subclass of groups $\Gamma$ that are subgroups of $\Sp(1)$, we can find a larger $\Z_2$-torus action with a fixed point.
\end{remark}

\begin{remark}
If $d$ is a prime in this example, and if no smaller prime divides $r$, then $S^{2dr-1}/\Gamma$ equipped with the $T^r$-action extending the given $\Z_2^r$-action provides an example of a positively curved manifold with non-cyclic fundamental group that has symmetry rank larger than any other known examples of this type (see Example 4.8 in \cite{Ken17}). Conjecturally, they form model spaces for the maximal symmetry rank problem in the context of non-cyclic fundamental groups (see \cite{W10,Ken17}).

Further examples of fundamental groups of positively curved manifolds are constructed in \cite{Sh,GSh,Baz}, and further obstructions in the case of torus symmetry are shown in \cite{Kha}. It would be interesting to analyze these results with $\Z_2$-torus symmetry in mind.
\end{remark}

\subsection{The Quarter-Maximal  ${\bf \Z_2}$-Fixed Point Symmetry Rank Case}

Here, we  think of this case in terms of  $S^{4m+3} \subset\HH^{m+1}$  for which $\Sp(m+1)\times_{\Z_2} \Sp(1)$ acts as follows:  $\Sp(m+1)$ acts on $\HH^{m+1}$ by matrix multiplication,  $\Sp(1)$ acts on the right by scalar multiplication, and  the normal $\Z_2$ subgroup generated by $(-I_{m+1}, -1) \in \Sp(m+1) \times \Sp(1)$ acts trivially. 
\begin{example}[{\bf
 ${\bf \HP^m}$, ${\bf \CP^{2m+1}/\Z_2}$, and ${\bf (4m+3)}$-dimensional space forms}]\label{1/4}
Choose a subgroup $H$ of the $\Sp(1)$ acting on the right, so that it contains $\{\pm 1\}$. 
Then $H$ is a subgroup of the Hopf action and so acts freely on $S^{4m+3}$. 
The quotient $S^{4m+3}/H$ then has positive curvature and an effective, isometric action by
	\[\Sp(m+1)\times_{\Z_2} (N_{\Sp(1)}(H)/H) \cong \Sp(m+1)/\Z_2 \times (N_{\Sp(1)}(H)/H).\]
Let $\Gamma \subseteq \Sp(1)$ be a finite subgroup of $\Sp(1)$. Then, taking $H$ to be one of $\Gamma$,  $\langle Q_8, S^1\rangle$, or $\Sp(1)$, we get an  $\Sp(m+1)/\Z_2$ action  
 on $S^{4m+3}/\Gamma$, $\CP^{2m+1}/\Z_2$, or $\HP^m$, respectively.

Note that $\Sp(m+1)$ contains a copy of $\Z_2^{m+1}$ consisting of diagonal matrices with $\pm1$ entries and a copy of $\Sp(1)\cong \{qI_{m+1}: q\in \Sp(1)\}$.
These subgroups commute with each other, and the latter contains $Q_8$. Projecting to $\Sp(m+1)/\Z_2$, we find commuting subgroups of $\Z_2^m \cong \Z_2^{m+1}/\Z_2$ and $Q_8/\Z_2 \cong \Z_2^2$, so indeed the three examples in the previous paragraph admit isometric, effective actions by $\Z_2^{m+2}$.

Additionally there are subgroups of corank two, three, or four, respectively, whose fixed-point set components are $S^3/\Gamma$, $\CP^3/\Z_2$, and $\HP^2$, respectively.
\end{example}

\appendix
\smallskip\section{Error correcting code bounds}\label{sec:Appendix}

In this appendix, we include a table with the error correcting codes needed for Lemma \ref{lem:ECC}. 
All values in this table can be found at 

\centerline{\scriptsize \url{http://www.codetables.de/BKLC/Tables.php?q=2&n0=1&n1=256&k0=1&k1=256}.}

We recall that $d(n, r)$ is the optimal upper bound of  the minimum Hamming weight of the elements in the image of an injective, linear map $\Z_2^r \to \Z_2^n$, and that setting $B = \{3, 4, 7, 11, 12, 23\}$, we define $\delta_B(n)$ to be $1$ for $n \in B$ and $0$ otherwise.
{\tiny
\begin{table}[htb]
\captionsetup{font=scriptsize}
\caption{ Minimum Weights for Error Correcting Codes for Part 1 of Lemma \ref{lem:ECC}} \label{table1} 
\begin{tabular}{|c|c|c|c||cc||c|c|c|c|}
\hline
\thead{$n$} & \thead{$r=\lceil\tfrac{n+1}{2} \rceil +\delta_B(n)$} & \thead{$d=\lfloor\tfrac{n+3}{4}\rfloor$} &\thead{$d(n, r)$}  && & \thead{$n$} & \thead{$r=\lceil\tfrac{n+1}{2} \rceil+\delta_B(n)$} & \thead{$d=\lfloor\tfrac{n+3}{4}\rfloor$ }& \thead{$d(n, r)$}\\
\hline
3      &  3     &1        & 1       &  &&   37     &         19       &   10           &  8  \\
\hline
  4    &    4    &       1  &   1      &&  & 38     &       20     &    10      &   8  \\
\hline
    5     &   3      &   2      &   2      &  & &  39        &      20      &  10         &  8--9 \\
 \hline
      6  &    4    &     2       &   2      &  &  & 40 &     21  &       10     &   8--9  \\
 \hline
        7&    5     &     2     &     2    &  &  &41&      21    &        11    &  9--10    \\
 \hline
        8&      5    &     2       &    2     &  &  & 42 &    22     &     11    &   8--10 \\
\hline
         9&     5      &    3        &    3     &  &  &  43& 22     &      11     &   9--10 \\
         \hline
      10   &      6   &    3     & 3        &  & &44&   23   &      11  &   9--10  \\
 \hline
       11 &     7    &     3    &    3     &  &  &  45& 23   &     12   & 10  \\
 \hline
       12 &     8    &    3          &   3 &  &  &  46 & 24 & 12 & 10   \\
 \hline
       13 &      7   &    4        &  4  &  &  & 47  & 24 & 12 & 11 \\
\hline
        14 &     8   &   4       &  4   &  &  & 48  & 25 & 12 & 10--11\\
\hline
        15 &      8 &   4    &  4  &  & & 49 &  25 &  13 &   10--12  \\
 \hline
      16  &       9  &    4     &      4   &  &  &50  &26& 13 &  10--12 \\
 \hline
       17 &        9      &    5   &   5   &  &  & 51 &26&13  &  10--12 \\
 \hline
        18&        10     &    5    &     4    &  &  & 52 & 27 & 13 & 10--12  \\
\hline
      19   &       10  &  5    &    5  &  &  &53&27  &14  & 10--12 \\
\hline
         20&       11  &  5   &   5  &  & & 54 & 28 &  14 & 10--12  \\
 \hline
 21       &        11  &   6         &   6  &  &  & 55 & 28 & 14 &  11--13  \\
 \hline
     22   &      12  &    6       &     6  &  &  &56 & 29 & 14  & 11-13 \\
 \hline
    23    &       13  &    6          &   6  &  &  &  57  &  29 & 15  &  12--14\\
\hline
        24 &      13   &   6      &    6   &  &  & 58 & 30  &  15 & 12--14\\
\hline
25         &      13    &   7    &  6  &  & &59 & 30  & 15 &  12--14 \\
 \hline
  26      &         14  &    7    &    6     &  &  &60 & 31 & 15 &  12--14    \\
 \hline
 27       &         14     &    7    &  7  &  &  &  61 & 31 & 16 & 12--14 \\
 \hline
  28      &        15   &    7           &   6  &  &  &62  & 32 & 16 & 12--14   \\
\hline
  29       &        15  &   8      &    7  &  &  &  63 & 32 & 16  &  12--15 \\
\hline
   30      &      16    &   8      &   7  &  & &  64 & 33 & 16 & 12--14 \\
 \hline
 31       &      16  &    8     &   8  & & &  65  & 33 & 17 &  12--15 \\
 \hline
   32     &       17        &   8     &   8  &  &  & 66 & 34 & 17  &  12--15 \\
 \hline
  33      &       17  &    9    &  8 & & 	&  67 & 34 & 17 &  12--16  \\
\hline
  34       &      18       &    9    &   8   	&& &68 & 35 & 17 & 12--16  \\
\hline
   35      &     18   &     9       &    8  	&&&  69& 35 & 18&  13--16\\
 \hline
  36      &     19        &   9   &    8  	&&&  70 & 36 & 18 &  13--16 \\
  
  \hline
\end{tabular}
\end{table}}
\
  
\bibliographystyle{alpha}

\begin{thebibliography}{999}

\bibitem{Ada60} J. F. Adams, {\em On the non-existence of elements of Hopf invariant one},
Ann. Math. {\bf  72} (1960), no. 1, 20--104.

\bibitem{AK20} M. Amann and L. Kennard, {\em Positive curvature and symmetry in small dimensions}, Comm. Contemporary Math. {\bf 22} no. 6 (2019), https://doi.org/10.1142/S0219199719500536.

\bibitem{Ash} R. B. Ash, {\em Information theory},
Dover Publications, Inc., New York, (1990).

\bibitem{Baz} Y.V. Baza\u{\i}kin, {\em A Manifold with Positive Sectional Curvature and Fundamental Group $\Z_3\oplus\Z_3$}, Sib. Math. J., {\bf 40} (1999), 834--836.

\bibitem{Ber} M. Berger, {\em Trois remarques sur les vari\'et\'es riemanniennes \`a courbure positive},
C. R. Math. Acad. Sci. Paris, {\bf 263} (1966), A76--A78.


\bibitem{Bre72} G. Bredon,  {\em Introduction to compact transformation groups}, Academic Press {\bf 48} (1972). 

\bibitem{Bor} A. Borel, {\em Seminar on Transformation Groups}, {\bf 46}, Annals of Mathematics Series, Princeton University Press (1961).
\bibitem{Fan08} F. Fang, {\em Finite isometry groups of $4$-manifolds with positive sectional curvature}, 
Math Z., {\bf 259} (2008), 643--656, DOI 10.1007/s00209-007-0242-0.

\bibitem{FG} F. Fang, K. Grove, {\em Reflection groups in non-negative curvature}, J. Diff. Geom., {\bf 102} no. 2 (2016), 179--205.


\bibitem{FR03} F. Fang and X. Rong, {\em Homeomorphism classification of positively curved manifolds with almost maximal symmetry rank}, Math. Ann. {\bf 332} (2005), 81--101.

\bibitem{FR04} F. Fang and X. Rong,   {\em Positively curved manifolds with maximal discrete symmetry rank}, Am. Jour.  of Math. {\bf 126}  (2004) 227--245.

\bibitem{FRW13} P. Frank, X. Rong, and Y. Wang, {\em Fundamental groups of positively curved manifolds with symmetry}, Math. Ann., {\bf 355}  (2013), 1425--1441.


\bibitem{GS} K. Grove and C. Searle,  {\em Positively curved manifolds with maximal symmetry rank}, Jour. of Pure and Appl. Alg. {\bf 91}  (1994), 137--142.

\bibitem{GSh} K. Grove, K. Shankar, {\em Rank two fundamental groups of positively curved manifolds}, J. Geom. Anal. {\bf 10} no. 4 (2000), 679--682.

\bibitem{Ham82} R. S. Hamilton, {\em Three-manifolds with positive Ricci curvature},
Differ. Geom. {\bf 17} (1982), 255--306.

\bibitem{Hat02} A. Hatcher, {\em Algebraic Topology},
Cambridge University Press, Cambridge, 2002.

\bibitem{Hic97} A. Hicks, {\em Group Actions and the Topology of Nonnegatively Curved 4-Manifolds},  Illinois J. Math. {\bf 41}  no. 3 (1997), 421--437.

\bibitem{HirMil} W. Hirsch and J. Milnor, 
{\em Some curious involutions of spheres}, Bull. Amer. Math. Soc.,
 {\bf 70}
      (1964),
     372--377.
    

\bibitem{Ken13} L. Kennard, {\em On the Hopf conjecture with symmetry}, Geom. Topol. {\bf 17} (2013) 563--593.

\bibitem{Ken17} L. Kennard, {\em Fundamental groups of manifolds with positive sectional curvature and torus symmetry},  J. Geom. Anal. {\bf 27} (2017), no. 4, 2894--2925.

\bibitem{KWW} L. Kennard, M. Wiemeler, and B. Wilking, {\em Splitting of torus representations and applications in the Grove symmetry program}, arXiv:2106.14723  (2021).

\bibitem{Kha} E. Khalili Samani, {\em Obstructions to free actions on {B}azaikin spaces},
Transform. Groups, {\bf 27} (2022), no. 4 1515--1532.

\bibitem{KL09}  J. H. Kim, H. K. Lee, {\em On non-negatively curved 4-manifolds with discrete symmetry},  Acta Math. Hungar., {\bf 125} (3) (2009), 201--208. DOI: 10.1007/s10474-009-8237-4. 

\bibitem{N} J. Nienhaus, {\em An improved four-periodicity theorem and a conjecture of Hopf with symmetry}, arXiv:2211.13151 (2022).



\bibitem{RS05} X. Rong and X. Su, {\em The Hopf Conjecture for Manifolds with Abelian Group Actions}, 
Comm. Contemp. Math. {\bf 7} no. 1 (2005)  121--136.


\bibitem{Sh} K. Shankar, {\em On the fundamental groups of positively curved manifolds}, J. Diff. Geom., {\bf 49} no. 1 (1998), 179--182.


\bibitem{Sm} P. A. Smith, {\em Transformations of finite period}, Ann. of Math. 
{\bf 39} no. 1 (1938), 127--164.

\bibitem{SW08} X. Su and Y. Wang, {\em The Hopf Conjecture for positively curved manifolds with discrete abelian group actions},
Differential Geometry and its Applications, vol. 26, issue 3  (2008), pp. 313--322.

\bibitem{Sug82} K. Sugahara, {\em The isometry group and the diameter of a Riemannian manifold with positive curvature},  Math. Japon., {\bf 27} (1982), 631--634.
 
 \bibitem{Wa} C. T. C. Wall, {\em Surgery on Compact Manifolds}, Mathematical Surveys and Monographs, Amer. Math. Society, {\bf 69}, 2nd edition (1998).
 
 \bibitem{W10} Y. Wang, {\em Fundamental groups of closed positively curved manifolds with
              almost discrete abelian group actions},
 Acta Mathematica Scientia. Series B. English Edition, {\bf 30}, no. 1 (2010),
   203--207 {\scriptsize\url{https://doi.org/10.1016/S0252-9602(10)60037-9}}.

\bibitem{Wil03} 
B. Wilking  {\em Torus actions on manifolds of positive sectional curvature},  
Acta Math., {\bf 191}  no. 2  (2003) 259--297. 


\bibitem{Wol} J.A. Wolf, {\em Spaces of constant curvature}, American Mathematical Society (2011).

\bibitem{Yan94} D. G. Yang  {\em On the Topology of Nonnegatively Curved Simply Connected 4-manifolds 
with Discrete Symmetry}, 
Duke Math. J., {\bf 74}  (1994),  531--545.


\end{thebibliography}

\end{document}